\documentclass[preprint,12pt,authoryear]{elsarticle}

\usepackage{amssymb}
\usepackage{amsmath}
\usepackage{amsthm}
\usepackage{graphicx}
\usepackage{natbib}
\usepackage{setspace}
\setcitestyle{numbers,square}
\usepackage{hyperref}
\numberwithin{equation}{section}%公式按章节编号
\numberwithin{figure}{section}%图表按章节编号
\newtheorem{theorem}{Theorem}[section] % 如果不采用章节号做前缀， 则不用[section]

\newtheorem{proposition}[theorem]{Proposition}
\newtheorem{lemma}[theorem]{Lemma}

\journal{Communications on Pure and Applied Analysis}

\begin{document}

\begin{frontmatter}

%% Title, authors and addresses

%% use the tnoteref command within \title for footnotes;
%% use the tnotetext command for theassociated footnote;
%% use the fnref command within \author or \affiliation for footnotes;
%% use the fntext command for theassociated footnote;
%% use the corref command within \author for corresponding author footnotes;
%% use the cortext command for theassociated footnote;
%% use the ead command for the email address,
%% and the form \ead[url] for the home page:
%% \title{Title\tnoteref{label1}}
%% \tnotetext[label1]{}
%% \author{Name\corref{cor1}\fnref{label2}}
%% \ead{email address}
%% \ead[url]{home page}
%% \fntext[label2]{}
%% \cortext[cor1]{}
%% \affiliation{organization={},
%%            addressline={}, 
%%            city={},
%%            postcode={}, 
%%            state={},
%%            country={}}
%% \fntext[label3]{}

\title{On A Parabolic Equation in MEMS with An External Pressure}
%% use optional labels to link authors explicitly to addresses:
%% \author[label1,label2]{}
%% \affiliation[label1]{organization={},
%%             addressline={},
%%             city={},
%%             postcode={},
%%             state={},
%%             country={}}
%%
%% \affiliation[label2]{organization={},
%%             addressline={},
%%             city={},
%%             postcode={},
%%             state={},
%%             country={}}
\author[author1]{Lingfeng Zhang}
\author[author1]{Xiaoliu Wang\corref{cor1}}
%\author[author1,author2]{Lingfeng Zhang\corref{cor2}} \author[author2]{Xiaoliu Wang\corref{cor1}}

\address[author1]{School of Mathematics, Southeast University, Nanjing, 211189, Jiangsu, People’s
Republic of China}
%%\affiliation[author1]{organization={School of Mathematics},%Department and Organization
  %%          addressline={Southeast University}, 
    %%        city={Nanjing},
      %%      postcode={211189}, 
        %%    state={Jiangsu},
          %%  country={People's Republic of China}}

\cortext[cor1]{Corresponding author: xlwang@seu.edu.cn}

\begin{abstract}
%% Text of abstract摘要部分
The parabolic problem $u_t-\Delta u=\frac{\lambda f(x)}{(1-u)^2}+P$ on a bounded domain $\Omega$ of $R^n$ with Dirichlet boundary condition models the microelectromechanical systems(MEMS) device with an external pressure term. In this paper, we classify the behavior of the solution to this equation. We first show that under certain initial conditions, there exists critical constants $P^*$ and $\lambda_P^*$ such that when $0\leq P\leq P^*$, $0<\lambda\leq \lambda_P^*$, there exists a global solution, while for $0\leq P\leq P^*,\lambda>\lambda_P^*$ or $P>P^*$, the solution quenches in finite time. The estimate of voltage $\lambda_P^*$, quenching time $T$ and pressure term $P^*$ are investigated. The quenching set $\Sigma$ is proved to be a compact subset of $\Omega$ with an additional condition, provided $\Omega\subset R^n$ is a convex bounded set. In particular, if $\Omega$ is radially symmetric, then the origin is the only quenching point. Furthermore, we not only derive the two-side bound estimate for the quenching solution, but also study the asymptotic behavior of the quenching solution in finite time. 
\end{abstract}

%%Graphical abstract
%\begin{graphicalabstract}
%\includegraphics{grabs}
%The parabolic problem $u_t-\Delta u=P+\frac{\lambda f(x)}{(1-u)^2}$ on a bounded domain $\Omega$ of $R^n$ with Dirichlet boundary condition models the microelectromechanical systems(MEMS) device with pressure term. In this paper, we classify the behavior of the solution to this equation. We first show that under certain initial conditions, there exists critical constants $P^*$ and $\lambda_P^*$ such that when $0\leq P\leq P^*$, $0<\lambda\leq \lambda_P^*$, the solution exists globally, while for $0\leq P\leq P^*,\lambda>\lambda_P^*$ or $P>P^*$, all the solutions quenches in finite time. The estimate of voltage $\lambda_P^*$, quenching time $T$ and pressure term $P^*$ are investigated. The quenching set $\Sigma$ is sure to be a compact subset of $\Omega$ with an additional condition, provided $\Omega\subset R^n$ is a convex bounded set. In particular, if $\Omega$ is radially symmetric, then the origin is the only quenching point. Furthermore, we not only derive the two-side bound estimate and the gradient estimate of the quenching rate, but also study the asymptotic behavior of the quenching solution in finite time. 
%\end{graphicalabstract}

%%Research highlights
%\begin{highlights}
%\item Research highlight 1
%\item Research highlight 2
%\end{highlights}

\begin{keyword}
%% keywords here, in the form: keyword \sep keyword
MEMS \sep parabolic type equation \sep quench \sep quenching set \sep quenching rate\sep asymptotic behavior

%% PACS codes here, in the form: \PACS code \sep code

%% MSC codes here, in the form: \MSC code \sep code
%% or \MSC[2008] code \sep code (2000 is the default)

\end{keyword}

\end{frontmatter}

%% \linenumbers

%% main text
\section{Introduction.}

MEMS (Micro Electromechanical System), also known as micro-systems or micro-machines, refers to high-tech devices with a size of a few millimeters or even smaller. This is a multidisciplinary research field gradually developed on the basis of microelectronic technology, including electronic engineering, mechanical engineering, physics, chemistry and other disciplines, and has developed rapidly in recent years. The key physical model of MEMS is the simple idealized electrostatic device shown in Figure 1, see \cite{ref11}. The upper part of this device consists of a thin and deformable elastic membrane that is held fixed along its boundary and lies above a rigid grounded plate. This elastic membrane is modeled as a dielectric with a small but finite thickness. The upper surface of the membrane is coated with a negligibly thin metallic conducting film. When a voltage $V$ is applied to the conducting film, the thin dielectric membrane deflects towards the bottom plate, and when $V$ is increased beyond a certain critical value $V^*$, which is known as pull-in voltage, the steady-state of the elastic membrane is lost and proceeds to quench or touch down at finite time.

\begin{figure}[htbp]
\centering
    \includegraphics[width=0.8\textwidth,height=0.4\textwidth]{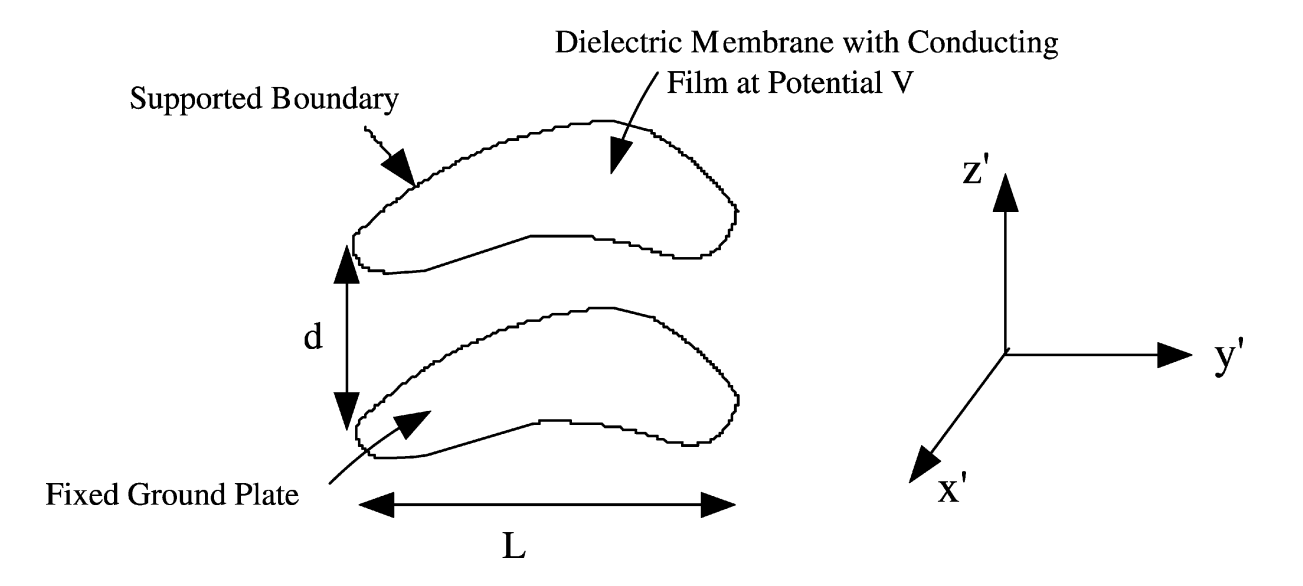}
    \caption{The simple electrostatic MEMS device.}
    \label{fig1}
    \end{figure}
    \vskip 10 pt

%%\ \ \ \ \ \ \ \ \ \ \ \ \ \ \ \ \ \ \ \ \ \ \ %%Figure 1 (see the attached file)

When designing almost any MEMS or NEMS device based on the interaction of electrostatic forces with elastic structures, the designers will always confront the “pull-in” instability. This instability refers to the phenomenon of quenching or touch down as we described previously when the applied voltage is beyond certain critical value $V^*$. A lot of research has been done in understanding and controlling the instability.  Most investigations of MEMS and NEMS have followed Nathanson’s work \cite{ref16} and used some sort of small aspect ratio approximation to simplify the mathematical model.  For an overview of the recent research on MEMS models, we refer to \cite{EGG,ref17,LW1} and references therein.

A simplified MEMS model is described by the following partial differential equation
\begin{equation}\label{eq without pressure term}
\begin{cases}u_t-\Delta u=\dfrac{\lambda f(x)}{(1-u)^2},(x,t)\in\Omega\times(0,T),\\ u(x,t)=0,(x,t)\in{\partial\Omega}\times(0,T),\\  u(x,0)=u_0(x),x\in\Omega,\end{cases}
\end{equation}
where $\Omega\subset R^n$ is a bounded region with smooth boundary $\partial\Omega$, 
$f(x)\ge 0$ represents the capacitance rate characteristic, the parameter $\lambda$ is the ratio of relative electrostatic force to relative elastic force, which is related to the applied voltage, $u(x,t)$ represents the deflection, which is related to both the time variable $t$ and the space variable $x$, and $u_0$ is the initial deflection. The study of (\ref{eq without pressure term}) starts from its stationary equation. It is shown in \cite{GG1} by Ghoussoub and Guo that there exists a pull-in voltage $\lambda^*$ such that:

(1) If $0<\lambda<\lambda^*$, there exists at least a global solution to the stationary equation of (\ref{eq without pressure term}); 

(2) If $\lambda>\lambda^*$, there is no solution to the stationary equation of (\ref{eq without pressure term}). 

For the parabolic equation in (\ref{eq without pressure term}), Ghoussoub and Guo \cite{ref11} proved that for the same $\lambda^*$ as above, it holds that when $\lambda>\lambda^*$, all solutions to (\ref{eq without pressure term}) quench in finite time, while when $\lambda\leq \lambda^*$, there exists a global solution which converges pointwisely to its unique maximal steady-state as $t\to\infty$ for some suitable initial data. More refined analysis of the quenching solution to (\ref{eq without pressure term}) are given in \cite{GG2,ref11} and references therein.  Later on, Ye and Zhou \cite{ref19} considered the relative elliptic and parabolic equations, which have a more general form. Guo-Souplet studied the dependence on $f(x)$ of quenching behavior interestingly in \cite{GS1}.

When the fringing term is incorporated into the MEMS model, people consider the following problem:
\begin{equation}\label{eq with fringing term}
\begin{cases}u_t-\Delta u=\dfrac{\lambda(1+\delta|\nabla u|^2)}{(1-u)^2},(x,t)\in\Omega\times(0,T),\\ u(x,t)=0,(x,t)\in\partial\Omega\times(0.T),\\ u(x,0)=u_0(x),x\in\Omega,\end{cases}
\end{equation}
where $\delta |\nabla u|^2$ represents the first-order correction of the fringing term. The stationary equation of (\ref{eq with fringing term}) are studied by Wei and Ye \cite{WY1} and their result is: for any fixed $\delta>0$ there exists a critical pull-in voltage $\lambda_\delta^*$ such that for $\lambda>\lambda_\delta^*$ there are no solutions; for $0<\lambda<\lambda_\delta^*$, there are at least two solutions; and for $\lambda=\lambda_\delta^*$, there exists a unique solution. For the problem (\ref{eq with fringing term}) with one-dimension space variable, Liu and Wang \cite{ref13} proved that for the same $\lambda_\delta^*$ as above,  it holds that for $\lambda>\lambda_\delta^*$, the solution to (\ref{eq with fringing term}) quenches in finite time, while for $0<\lambda<\lambda_\delta^*$, there exists a global solution to (\ref{eq with fringing term}). For particular initial data, they showed the property of finitely many quenching points. In \cite{ref14}, Luo and Yau  carried a similar classification of the behaviour of solutions to (\ref{eq with fringing term}) with higher dimensional space variable, and obtained not only the compactness of the quenching set, but also the estimates on the quenching time, the pull-lin voltage $\lambda_\delta^*$ and the quenching rate. The asymptotic behaviour of the quenching solution are also shown under some additional assumption. Also please see references \cite{GHWWY,IS1,Miyasita,PX1,WangQ1}  for the relative study.

Recently, Guo, Zhang and Zhou \cite{ref9} considered the stationary equation to MEMS with an external pressure :

\begin{equation}\tag{$E_{\lambda,P}$}
\begin{cases}-\Delta u=\dfrac{\lambda f(x)}{(1-u)^2}+P,(x,t)\in\Omega\times(0,T),\\ u(x,t)=0,x\in{\partial\Omega},\\  u(x,0)=u_0(x),x\in\Omega,\end{cases}\label{eq:static equation}
\end{equation}
where $\lambda\ge 0$ denotes the voltage, $P>0$ denotes the pressure term, and $\Omega$ is a region with smooth boundary. In \cite{ref9}, the authors proved that when $P\ge P^*$ and $\lambda>0$, (\ref{eq:static equation}) has no positive classical solution. When $0\leq P<P^*$, there exists 
a $\lambda_P^*$, such that if $0<\lambda\leq \lambda_P^*$, then (\ref{eq:static equation}) has at least one solution; if $\lambda>\lambda_P^*$, then (\ref{eq:static equation}) has no solution.

\vskip 10 pt

In this paper, we study the parabolic version of the problem ($E_{\lambda,P}$), that is
\begin{equation}\tag{$P_{\lambda,P}$}
\begin{cases}u_t-\Delta u=\dfrac{\lambda f(x)}{(1-u)^2}+P,(x,t)\in\Omega\times(0,T),\\ u(x,t)=0,(x,t)\in{\partial\Omega}\times(0,T),\\  u(x,0)=u_0(x),x\in\Omega,\end{cases}\label{eq:main equation}
\end{equation}
where $\Omega$ is an open set of $R^n$, $\lambda>0$, $P>0$, $T>0$; $f(x),u_0(x)\in C^2$, $f(x)\ge c_0$, $c_0>0$; $u_0(x)\ge 0$, and $\Delta u_0(x)+\frac{\lambda f(x)}{(1-u)^2}+P\ge 0$. 

Our first aim is to  classify the solution's behaviour. Then by following the lines of  Luo and Yau \cite{ref14}, we try to understand various quenching properties and characterize the asymptotic behaviour of quenching solutions. The first two theorems stated in the following are about the instability of (\ref{eq:main equation}).

\begin{theorem}\label{thm 1.1}
We consider the case of $0\leq P\leq P^*, 0<\lambda\leq \lambda_P^*$. Let $u_s$ be an upper solution to (\ref{eq:static equation}) such that $0\leq u_0\leq u_s<1$. Then there exists a unique global solution $u(x,t)$ of (\ref{eq:main equation}), which monotonically converges as $t\to \infty$ to the minimal solution to (\ref{eq:static equation}) in the sense of $C^2$.
\end{theorem}

When the parameter $\lambda$ or the pressure term $P$ exceeds the critical value, the solution to (\ref{eq:main equation}) will quench in finite time, that is,

\begin{theorem}\label{thm 1.2}
If $0\leq P\leq P^*,\lambda>\lambda_P^*$ or $P>P^*$, $u_0(x)$ satisfies $0\leq u_0<1$, there exists a finite time $T=T(\lambda,P)$ such that the unique solution $u(x,t)$ to (\ref{eq:main equation}) quenches at $T$.
\end{theorem}

If $u(x,t)$ touches $1$ in finite time, we say that the solution to (\ref{eq:main equation}) quenches, and the set of quenching points is called quenching set, that is
$$ \Sigma=\{x\in\overline{\Omega}\colon\exists(x_n,t_n)\in\Omega_T,s.t.x_n\to x,t_n\to T,u(x_{n},t_n)\to1\},$$
and the precise definition of quenching time $T$ is 

$$T=\sup\bigl\{t>0\colon\bigl||u(\cdot \ ;s)\bigr|\bigl|_\infty<1,\forall s\in[0,t]\bigr\}.$$

Actually, with the similar argument in \cite{ref14}, we show that

\begin{proposition}\label{prop 1.3}
Let $\phi$ be any non-negative $C^2$ function, $\phi\not\equiv 0$ in $\Omega$ and $\phi =0$ on $\partial \Omega$. Then the quenching time T of (\ref{eq:main equation}) satisfies the following upper bound estimate for sufficiently large $\lambda$,

\begin{equation}
T \leq \frac{\|\phi\|_{1}}{3 \lambda\|f \phi\|_{1}-\|\Delta \phi\|_{1}},
\end{equation}
\end{proposition}
where $\left \|\cdot  \right \|_1 $ is the $L^1$ norm. This implies that

$$T \lesssim \frac{1}{\lambda},$$
if $\lambda \gg 1$.  The notation $a\lesssim b$ means there exists some constant $C>0$ such that $a\leq Cb$. We assume $\Omega\subset R^n$ is a convex bounded domain. By the moving plane argument, we assert that the quenching set is a compact subset of $\Omega$. And if $\Omega=B_R$, the ball centered at the origin with the radius $R$, then the quenching solution is radially symmetric \cite{ref7} and the only quenching point is the origin.

\begin{theorem}\label{thm 1.4}
Let $\Omega\subset R^n$ be a convex set, $u(x,t)$ be a solution to (\ref{eq:main equation}) which quenches in finite time, and $\frac{\partial f}{\partial n}\leq 0$. Then the quenching set $\Sigma$ is a compact subset of $\Omega$.
\end{theorem}

\begin{theorem}\label{thm 1.5}
Set $\ \Omega=B_R(0)$. If $\lambda>\lambda_P^*$, then the solution only quenches at $r=0$, that is, the origin is the only quenching point.
\end{theorem}

To understand the quenching behavior of the finite time quenching solution to (\ref{eq:main equation}), We study upper bound estimate, lower bound estimate, non-degeneracy and asymptotic behavior of the solution which quenches in finite time. We begin with the upper bound estimate.

\begin{lemma}\label{lem 1.6}
Let $\Omega \subset R^n$ be a bounded convex domain and $u(x,t)$ be a solution to (\ref{eq:main equation}) which quenches in finite time. Then there exists a bounded positive constant $M>0$ such that $M(T-t)^{\frac{1}{3}}\leq 1-u(x,t)$ holds for all $(x,t)\in \Omega_T$, and when u quenches, $u_t \to +\infty$.
\end{lemma}

%We show in this paper that under certain conditions, the solution quenches in finite time $T$ with the rate
%$$1-u(x,t)\sim (3\lambda f(a) (T-t))^\frac{1}{3},$$
%as $t\to T^-$, provided $\Omega \in R$ or $\Omega\in R^n$, $n\ge 2$, is a radially symmetric domain.

To study the asymptotic behavior of the quenching solution, we make the similar transformation at some point $a\in \Omega_\eta$ as in \cite{ref11} and \cite{ref5}:
\begin{equation}
y=\dfrac{x-a}{\sqrt{T-t}},s=-\log(T-t),u(x,t)=1-(T-t)^{\frac{1}{3}}w_a(y,s), 
\end{equation}
where $\Omega_\eta=\left \{ x\in\Omega:dist(x,\partial \Omega)>\eta \right \}$, for some $\eta \ll 1$. First, the point $a$ can be identified as a nonquenching point if $w_a(y,s)\to \infty$, as $s\to +\infty$ uniformly in $|y|\leq C$, for any constant $C>0$. This is called the nondegeneracy phenomena in \cite{ref6}. This property is not difficult to derive. It follows immediately from the comparison principle.

The basis of the method, the similarity variables in \cite{ref11}, is the scaling property of $u_t-\Delta u=\frac{\lambda}{(1-u)^2}$: the fact that if $u$ solves it near $(0,0)$, then so do the rescaled functions
\begin{equation}
1-u_\gamma (x,t)=\gamma^{-\frac{2}{3}}\left [ 1-u(\gamma x,\gamma^2 t) \right ],  
\end{equation}
for each $\gamma>0$.  If $(0,0)$ is a quenching point, then the asymptotic of the quenching are encoded in the behavior of $u_\gamma$ as $\gamma\to 0$. Hence, the study of the asymptotic behavior of $u$ near the singularity is equivalent to understanding the behavior of $w_a(y,s)$, as $s\to +\infty$, which satisfies the equation:
\begin{equation}
\dfrac{\partial w_a}{\partial s}=\dfrac{1}{3}w_a-\dfrac{1}{2}y\nabla w_a+\Delta w_a-\dfrac{\lambda f\left(a+ye^{-\frac{s}{2}}\right)}{w_a^2}-Pe^{-\frac{2}{3}s}.  
\end{equation}

\begin{theorem}\label{thm 1.7}
Let $\Omega \subset R^n$ be a bounded convex domain and $u(x,t)$ be a solution to (\ref{eq:main equation}) which quenches in finite time, then $w_a(y,s)\to w_\infty (y)$, as $s\to \infty$ uniformly on $|y|\leq C$, where $C>0$ is any bounded constant, and $w_\infty (y)$ is bounded positive solution to
\begin{equation}
    \triangle w-\dfrac12y\cdot\nabla w+\dfrac13w-\frac{\lambda f(a+ye^{-\frac{s}{2}})}{w^2}=0
\end{equation}
in $R^n$. Moreover, if $\Omega\in R$ or $\Omega\in R^n$, $n\ge 2$, is a convex bounded domain, then we have
$$\lim\limits_{t\to T^-}(1-u(x,t))(T-t)^{-\frac{1}{3}}=(3\lambda f(a))^{\frac{1}{3}}$$
uniformly on $|x-a|\leq C\sqrt{T-t}$ for any bounded constant C.
\end{theorem}
\section{Global existence or quenching in finite time.}
\subsection{Global existence.}

\begin{theorem}[Global existence]\label{thm 2.1}
We consider the case of $0\leq P\leq P^*, 0<\lambda\leq \lambda_P^*$. Let $u_s$ be an upper solution to (\ref{eq:static equation}) such that $0\leq u_0\leq u_s<1$ , then there exists a unique global solution $u(x,t)$ of (\ref{eq:main equation}), which monotonically converges as $t\to \infty$ to the minimal solution to (\ref{eq:static equation}) in the sense of $C^2$.
\end{theorem}

A lemma is given about (\ref{eq:static equation}) \cite{ref9}, that is

\begin{lemma}\label{lem 2.2}
Let $f(x)$ satisfy $0 \leq f \in C^{\varepsilon}(\bar{\Omega}), \varepsilon \in(0,1] $ and $ f \not \equiv 0 $,
\begin{enumerate}[(1)]
    \item If $P\ge P^*$, as long as $\lambda>0$, there is no positive classical solution to (\ref{eq:main equation});
    \item If $0\leq P<P^*$, $\exists \lambda_P^*=\lambda_P^*(f,\Omega)$, such that
\begin{equation}\label{eq estiamte for lem 2.2}
\frac{4}{27\left(P^{*}\right)^{2} \sup f}\left(P^{*}-P\right)^{3} \leq \lambda_{P}^{*} \leq \frac{|\Omega|-P \int_{\Omega} \Phi d x}{\int_{\Omega} \Phi f d x}
\end{equation}

where $P\int_\Omega \Phi dx<\frac{P}{P^*}\left | \Omega \right |  $, $\Phi$ is the unique positive solution to $-\Delta \Phi=1 $ in $\Omega$; $\Phi=0$ on $\partial \Omega$, then
\begin{enumerate}[(a)]
\item there exists at least one positive solution when $0\leq \lambda<\lambda_P^*$;
\item there is no positive solution when $\lambda>\lambda_P^*$;
\end{enumerate}
    \item $\lambda_P^*$ is non-increasing with respect to $P$.
\end{enumerate}
\end{lemma}
\begin{proof}
Since $u_s$ is an upper solution to (\ref{eq:static equation}), we have $u(x,t)\leq u_s<1$, that is, the nonlinear term $\frac{\lambda f(x)}{(1-u)^2}$  has no singularity, so we can easily obtain the $\tilde{C^{1,\alpha}}$ boundedness of $u(x,t)$ when $0<\alpha<1$. It can be known from Schauder theory \cite{ref4} that (\ref{eq:main equation}) has a $\tilde{C^{1,\alpha}}$ bounded unique global solution $u(x,t)$.

Since $u(x,t)$ is bounded, $0\leq u(x,t)\leq u_s<1$, we differentiate $u_t-\Delta u=P+\frac{\lambda f(x)}{(1-u)^2}$ with respect to $t$, and let $w=u_t$, then we have $u_{tt}-\Delta u_t=\frac{2\lambda f(x)}{(1-u)^3}u_t$, that is

$$\left\{\begin{aligned}
w_{t}-\Delta w=\frac{2 \lambda f(x)}{(1-u)^{3}} w,(x, t) \in \Omega \times(0, T), \\
w(x, t)=0,(x, t) \in \partial \Omega \times(0, T), \\
w(x, 0)=P+\frac{\lambda f(x)}{\left(1-u_{0}\right)^{2}}+\Delta u_{0} \geq 0, x \in \Omega .
\end{aligned}\right.$$

The non-negativity of $\frac{2\lambda f(x)}{(1-u)^3}$ is obvious, and it is locally bounded. Obviously $f(x)$ is $\varepsilon -H\ddot{o} lder$ continuous, then $f(x)$ is uniformly continuous, and $f(x)$ is bounded on $\Bar{\Omega}$, provided $\Bar{\Omega}$ is a closed set. Let $|f(x)|\leq M_0$, then $\left | \frac{2\lambda f(x)}{(1-u)^3}  \right |\leq \frac{2\lambda_P^* M_0}{(1-u_s)^3} $ is locally bounded.

According to the maximum principle of parabolic equation \cite{ref1}, $u_t=w>0, (x,t)\in \Omega \times (0,T)$ or $w\equiv 0$. When $w\equiv 0$, $u(x,t)$ is independent of $t$, $u(x,t)=u_s(x)\ge u_0(x)$ does not satisfy the initial value condition of (\ref{eq:static equation}), so $u_t=w>0, (x,t)\in \Omega \times (0,T)$. Then $\lim_{t \to \infty}u(x,t)=u_{ss}(x)$, and $1>u_s(x)\ge u_{ss}(x)>0 $ holds on $\Omega$, according to the monotone bounded principle.

Next, we claim that $u_{ss}(x)$ is a solution to (\ref{eq:static equation}). Let us consider $u_1(x)$ satisfying

$$\left\{\begin{aligned}
-\Delta u_{1}=\frac{\lambda f(x)}{\left(1-u_{s s}\right)^{2}}+P, x \in \Omega, \\
u_{1}(x)=0, x \in \partial \Omega.
\end{aligned}\right.$$

Let $\Bar{u}(x,t)=u(x,t)-u_1(x)$, then $\Bar{u}(x,t)$ satisfies

$$\left\{\begin{aligned}
\bar{u}_{t}-\Delta \bar{u}=\lambda f(x)\left[\frac{1}{(1-u)^{2}}-\frac{1}{\left(1-u_{ss}\right)^{2}}\right],(x, t) \in \Omega \times(0, T), \\
\bar{u}(x, t)=0,(x, t) \in \partial \Omega \times(0, T), \\
\bar{u}(x, 0)=-u_{1}(x), x \in \Omega .
\end{aligned}\right.$$

The right-hand side tends to zero in $L^2(\Omega)$, as $t\to \infty$, which follows from
$$\begin{aligned}
\left|\lambda f(x)\left[\frac{1}{(1-u)^2}-\frac{1}{\left(1-u_{ss}\right)^2}\right]\right|\le\frac{2\lambda_p^*M_0|u-u_{ss}|}{\left(1-u_\xi\right)^3}\le\frac{2\lambda^*_pM_0|u-u_{ss}|}{\left(1-u_s\right)^3}\\
=:\bar{M}(u_s,\lambda_p^*,M_0)|u-u_{ss}|\to 0.
\end{aligned}$$
where $0<u\leq u_\xi \leq u_{ss}\leq u_s<1$. A standard eigenfunction expansion implies that $\Bar{u}(x,t)$ converges to zero in $L^2(\Omega)$ as $t\to \infty$. That is, $u(x,t)\to u_1(x)$, as $t\to \infty$. Combined with the fact that $u(x,t)\to u_{ss}(x)$  pointwisely as $t\to \infty$, we deduce that $u_1(x)=u_{ss}(x)$ almost everywhere in the sense of $L^2$, that is, $u_{ss}$ is also a solution to (\ref{eq:static equation}).The minimal property of $u_\xi$ yields that $u_\xi=u_{ss}$ in $\Omega$, from which follows that for every $x\in \Omega$, we have $u(x,t)\uparrow u_\xi(x)$, as $t\to \infty$. 
\end{proof}

\subsection{Finite-time quenching.}
\begin{theorem}[Finite-time quenching]\label{thm 2.3}
If $0\leq P\leq P^*,\lambda>\lambda_P^*$ or $P>P^*$, $u_0(x)$ satisfies $u_0\in \left [ 0,1 \right ] $, $c<1$, there exists a finite time $T=T(\lambda,P)$ at which the unique solution $u(x,t)$ to (\ref{eq:main equation}) quenches.
\end{theorem}

\begin{proof}
By contradiction, let $\lambda >\lambda_P^*$, $0\leq P\leq P^*$ and suppose there exists a solution $u(x,t)$ of (\ref{eq:main equation}) in $\Omega\times (0,\infty)$. Then for any $\eta>1$, $v=\frac{u}{\eta}\leq \frac{1}{\eta}<1$, $v=\frac{u}{\eta}< u$, $\frac{P}{\eta}<P$, and $v_{t}-\Delta v=\frac{1}{\eta}(u_{t}-\Delta u)=\frac{1}{\eta}(P+\frac{\lambda f(x)}{(1-u)^2})$, so $v$ is an upper solution to (\ref{eq:main equation}).

Therefore, we obtain a global solution $\Bar{u}$ of $P_{\frac{\lambda}{\eta},\frac{P}{\eta}}$ according to Theorem \ref{thm 2.1}, and $\Bar{u}\leq v\leq \frac{1}{\eta}<1$, $\lim_{t \to \infty}\bar{u}(x,t)=w  $ is a solution to $E_{\frac{\lambda}{\eta},\frac{P}{\eta}}$. When $\eta \to 1^+$, $\frac{\lambda}{\eta}>\lambda_P^*$, this contradicts the the nonexistence result in Lemma \ref{lem 2.2}. The same is true for the case of $P>P^*$.
\end{proof}

\section{Estimates for the pull-in voltage, the pressure term and the finite quenching time.}
\subsection{Estimates for the finite quenching time.}
\begin{proposition}[Upper bound of $T$]\label{prop 3.1}
Let $\phi$ be any non-negative $C^2$ function such that $\phi\not\equiv 0$ and $\phi =0$ on $\partial \Omega$, then for $\lambda$ large enough, the quenching time $T$ for the solution to (\ref{eq:main equation}) satisfies the following upper bound estimate

\begin{equation}
T \leq \frac{\|\phi\|_{1}}{3 \lambda\|f \phi\|_{1}-\|\Delta \phi\|_{1}},
\end{equation}
where $\left \| \cdot  \right \| _1$ is the $L^1$ norm of $\cdot$ in $\Omega$.
\end{proposition}
\begin{proof}
    Using $\phi (1-u)^2$ as the test function to (\ref{eq:main equation}) and integrating on $\Omega$, we can see
   $$\begin{aligned}
\left(\int_{\Omega} \frac{1}{3}\left[1-(1-u)^{3}\right] \phi d x\right)_{t}=\int_{\Omega}(1-u)^{2} \phi\left(\frac{\lambda f(x)}{(1-u)^{2}}+P+\Delta u\right) d x \\
%%=\int_{\Omega}(1-u)^{2} \phi \Delta u d x+\int_{\Omega}(1-u)^{2} \phi P d x+\int_{\Omega} \phi \lambda f(x) d x \\
\geq \int_{\Omega} \phi \lambda f(x) d x+\left.\nabla u \cdot \phi(1-u)^{2}\right|_{\partial \Omega}-\int_{\Omega} \nabla u \nabla\left(\phi(1-u)^{2}\right) d x \\
%%=\int_{\Omega} \phi \lambda f(x) d x-\int_{\Omega} \nabla u \nabla \phi(1-\mathrm{u})^{2}+\nabla u \phi \nabla(1-u)^{2} d x \\
%%=\int_{\Omega} \phi \lambda f(x) d x-\int_{\Omega} \nabla\left(\frac{1}{3}\left(1-(1-u)^{3}\right)\right) \nabla \phi d x-\int_{\Omega} \nabla u \phi \nabla(1-u)^{2} d x \\
%%=\int_{\Omega} \phi \lambda f(x) d x+\int_{\Omega} \Delta \phi \frac{1}{3}\left(1-(1-u)^{3}\right) d x-\int_{\Omega} \nabla u \phi \nabla(1-u)^{2} d x \\
\geq \int_{\Omega} \phi \lambda f(x) d x+\frac{1}{3}\int_{\Omega} \Delta \phi \left(1-(1-u)^{3}\right) dx,
\end{aligned}$$
since $\int_{\Omega} \nabla u \phi \nabla(1-u)^{2} d x=\int_{\Omega} -2|\nabla u|^2 \phi(1-u) dx\leq 0$.

Therefore, for any $t<T$, integrating from 0 to $t$, we obtain that
$$\begin{aligned}
\frac{1}{3} \int_{\Omega} \phi d x \geq \int_{\Omega} \frac{1}{3}\left[1-(1-u)^{3}\right] \phi dx\geq \int_{0}^{t} \int_{\Omega} \Delta \phi \frac{1}{3}\left(1-(1-u)^{3}\right) dx dt\\
+t \int_{\Omega} \phi \lambda f(x) d x\geq \lambda t \int_{\Omega} \phi f(x) d x-\frac{1}{3} t \int_{\Omega}|\Delta \phi| dx, 
\end{aligned}$$
since $\Delta \phi \geq-|\Delta \phi|$, and $\frac{1}{3}\left(1-(1-u)^{3}\right) \leq 1$. Then let $t\to T^-$, namely 
\begin{equation}
T \leq \frac{\frac{1}{3} \int_{\Omega} \phi d x}{\lambda \int_{\Omega} \phi f(x) d x-\frac{1}{3} \int_{\Omega}|\Delta \phi| d x}=\frac{\|\phi\|_{1}}{3 \lambda\|f \phi\|_{1}-\|\Delta \phi\|_{1}} .
\end{equation}
Therefore, we obtain the upper bound estimate of $T$.
\end{proof}

We compare the quenching time $T=T(\lambda,P)$ with different $\lambda$ and different $P$ respectively:

\begin{proposition}\label{prop 3.2}
\begin{enumerate}[(1)]
    \item Suppose $u_1(x,t)$ and $u_2(x,t)$ are solutions of (\ref{eq:main equation}) with $\lambda=\lambda_1$ and $\lambda_2$, respectively. The corresponding finite quenching times are denoted by $T_{\lambda_1}$ and $T_{\lambda_2}$, respectively. If $\lambda_1>\lambda_2$, then $T_{\lambda_1}<T_{\lambda_2}$;
    \item Suppose $u_1(x,t)$ and $u_2(x,t)$ are solutions of (\ref{eq:main equation}) with $P=P_1$ and $P_2$, respectively. The corresponding finite quenching times are denoted by $T_{P_1}$ and $T_{P_2}$, respectively. If $P_1>P_2$, then $T_{P_1}<T_{P_2}$.
\end{enumerate}
\end{proposition}
\begin{proof}
Let $\hat{u}=u_1-u_2 $, then $\hat{u}|_{\partial \Omega}(x,t)=\hat{u}(x,0)=0$, and
$$\begin{aligned}
\hat{u}_t-\Delta \hat{u}=\dfrac{\lambda_1 f(x)}{(1-u_1)^2}+P-\dfrac{\lambda_2 f(x)}{(1-u_2)^2}-P\\
>\lambda_2 f(x)\left [ \frac{1}{(1-u_1)^2}-\frac{1}{(1-u_2)^2}   \right ]=\dfrac{2\lambda_2 f(x)\hat{u}}{(1-u_\xi)^3}. 
\end{aligned}$$

Hence, $\hat{u}>0$ on $\Omega \times (0,\min T_{\lambda_i} )$ according to the maximum principle \cite{ref1}, namely $u_1>u_2$, so $T_{\lambda_1}<T_{\lambda_2}$. The same is true for the case of $P=P_1$ and $P_2$.
\end{proof}

\subsection{Estimates for the pull-in voltage and the pressure term.}
An upper bound of $\lambda^*$ is given in \cite{ref8}, namely
\begin{equation}
\lambda^*\leq \frac{4}{27}\mu_0,
\end{equation}
where $\mu_0>0$ is the first eigenvalue of $-\Delta \phi_0=\mu_0 \phi_0$, $x\in\Omega$ with Dirichlet boundary condition. We shall derive an upper bound for $\lambda_P^*$ and $P^*$:
\begin{proposition}[Upper bound for $\lambda_P^*$ and $P^*$]\label{prop 3.3}
$\lambda_P^*$ of (\ref{eq:main equation}) has an upper bound, namely $\lambda_P^*\leq \lambda_1:=\frac{\mu_0-P}{c_0}$ and P also has an upper bound, namely $P^*\leq \mu_0$.
\end{proposition}
\begin{proof}
 Let $\mu_0>0$ and $\phi_0>0$ be the first pair of $-\Delta \phi_0=\mu_0 \phi_0$, $x\in\Omega$ with Dirichlet boundary condition. We multiply the stationary equation of (\ref{eq:static equation}) by $\phi_0$, integrate the equation over $\Omega$, and use Green’s identity to get
\begin{equation}
\begin{aligned}
\int_{\Omega} \phi_{0}\left(\Delta u+\frac{\lambda f(x)}{(1-u)^{2}}+P\right) dx=\int_{\Omega} \phi_{0}\left(\frac{\lambda f(x)}{(1-u)^{2}}+P\right) dx\\
+\int_{\Omega} \Delta \phi_{0} u dx=\int_{\Omega} \phi_{0}\left(\frac{\lambda f(x)}{(1-u)^{2}}+P-\mu_{0}\right) dx.
\end{aligned}
\end{equation}

If $P^*>\mu_0$, there exists some $P_0>\mu_0$ such that $\frac{\lambda f(x)}{(1-u)^2}+P_0-\mu_0>0$, then there is no solution to the stationary equation of (\ref{eq:static equation}), and so to (\ref{eq:main equation}), so $P^*\leq \mu_0$.

If $P^*\leq \mu_0$, and $\lambda>{\frac{\mu_{0}-P}{c_{0}}},{\frac{\lambda f(x)}{(1-u)^{2}}}+P-\mu_{0}\geq{\frac{\mu_{0}{-P}}{(1-u)^{2}}}+P-\mu_{0}>0$, then (\ref{eq:static equation}) also has no solution, and so to (\ref{eq:main equation}), then $\lambda_P^*\leq \lambda_1:=\frac{\mu_0-P}{c_0}$.
\end{proof}

\section{Quenching set.}
The quenching set is defined by $\Sigma=\{x\in{\overline{{\Omega}}}\colon\exists(x_{n},t_{n})\in\Omega_{T}, s.t.\  x_{n}\to x,t_{n}\to T,u(x_{n},t_n)\to1\}$, where $\Omega$ is a bounded convex subset. In this section, it is followed by the moving plane argument that the quenching set of any finite-time quenching solution to (\ref{eq:main equation}) is a compact subset of $\Omega$.

\begin{theorem}[Compactness of the quenching set]\label{thm 4.1}
Let $\Omega\subset R^n$ be a convex set, $u(x,t)$ be a solution to (\ref{eq:main equation}) which quenches in finite time $T$, and $\frac{\partial f}{\partial n}\leq 0$. Then the quenching set $\Sigma$ is a compact subset of $\Omega$.
\end{theorem}

\begin{proof}
Let us denote $x=(x_1,x')\in R^n$, $x'=(x_2,x_3,...,x_n)\in R^{n-1}$, take any point $y_0\in \partial \Omega$ and without loss of generality suppose that $y_0=0$ (or we can translate the whole $\Omega$), and that the half plane $\left \{ x_1>0 \right \} $ is tangent to $\Omega$ at $y_0$ (that is the tangent plane of $\Omega$ at $y_0$).

Let $\Omega_\alpha^+=\Omega\cap \left \{ x_1>0 \right \}$, $\alpha<0$, $\left | \alpha \right |$ small enough, and
$$\begin{aligned}
\Omega_\alpha^+=\{x=(x_1,x')\in R^n\colon(x_1,x’)\in\Omega,x_1>\alpha\},\\ \Omega_\alpha^-=\{x=(x_1,x')\in R^n\colon(2\alpha-x_1,x’)\in \Omega_\alpha^+\},
\end{aligned}$$
where $\Omega_\alpha^-$ is the reflection of $\Omega_\alpha^+$ on $\left \{ x_1=\alpha \right \} $.

From the derivative maximum principle \cite{ref18}, we can observe that for any $(x,t)\in \Omega \times (0,T),u(x,t)\ge 0$, and $\frac{\partial u}{\partial \gamma}(t_0)<0$ on $\partial \Omega$ for some $t_0\in (0,T)$, where $\gamma$ is the outer normal vector of $\partial \Omega$.

We consider $\bar{u}(x,t)=u(x_1,x',t)-u(2\alpha-x_1,x',t),x\in\Omega_{\alpha}^{-}$, then $\Bar{u}$ satisfies 
$$\begin{aligned}
\overline{u}_t(x,t)-\Delta\overline{u}(x,t)=P-P+\frac{\lambda f(x)}{\left(1-u(x_1,x',t)\right)^2}-\frac{\lambda f{(x)}}{\left(1-u(2\alpha-x_1,x',t)\right)^2}\\
=\lambda \Bar{u}(x,t)f(x)\frac{2}{(1-u_\xi)^3}=:\lambda c(x,t)\bar{u}(x,t),
\end{aligned}$$
where $c(x,t)$ is obviously a bounded function. 

On $\left \{ x_1=\alpha \right \} $, $\bar{u}(x,t)=u(\alpha)-u(\alpha)=0 $ and $u(2\alpha-x_1,x',t)=0,(2\alpha-x_1,x)\in \partial \Omega$, where $\partial\Omega_\alpha^{-}\cap\{x_1<\alpha\}\times(0,T]=\{(x,t)=(x_1,x',t)\colon(2\alpha-x_1,x')\in\partial\Omega_\alpha^+,x_1<\alpha,0<t\leq T\}.$

If $\alpha$ is small enough, by moving plane argument, we claim that $\bar{u}(x,t_0)=u(x_1,x',t_0)-u(2\alpha-x_1,x',t_0)\ge 0$, for $x\in \Omega_\alpha^-$. We prove that by contradiction. If there exists $x^0\in\Omega_{\alpha}$ such that $\bar{u}(x^0,t_0)<0$, then we have the positive function $\phi$ satisfying $-\Delta\phi=\lambda\mu(x)\phi(x)$, $\mu(x)<c(x)$ on $\Bar{\Omega}$. Let $g(x)=\frac{\Bar{u}(x)}{\phi(x)}$, then it is easy to see that there exists a point of $\Bar{u}(x)<0$ in $\Omega$ by $\phi(x)>0$. Let $x^0\in \Omega$ be a minimum point of $\Bar{u}(x)$, and  $\Bar{u}(x^0)<0$, by direct calculation we can know that
%%移动平面法的证明过程（看情况可以省略）
\begin{equation}
-\Delta g=2\nabla g\dfrac{\nabla\phi}{\phi}+\dfrac{1}{\phi}\left(-\Delta\bar{u}+\dfrac{\Delta\phi}{\phi}\bar{u}\right).
\end{equation}

Since $x^0$ is the minimum point, then we have $-\Delta g(x^0)\leq0$ and $\nabla g(x^0)=0$. However, 
$$-\Delta\bar{u}+\frac{\Delta\phi}{\phi}\bar{u}\ge -\Delta\overline u(x^0)+\lambda\mu(x^0)\overline u(x^0)>-\Delta\overline{u}(x^0)+\lambda c(x^0)\bar u(x^0)=0,$$
this contradicts $-\Delta g(x^0)\leq0$.

Applying the maximum principle \cite{ref18} and the boundary point lemma, we can see that
\begin{equation}
\begin{cases}\bar{u}(x,t)>0\ in\ \bar{\Omega}_{\alpha}\times(t_0,T),\\ \dfrac{\partial\bar{u}}{\partial\gamma_1}=-2\dfrac{\partial\overline{u}}{\partial x_1}>0\ on\ \{x_1=\alpha\}.\end{cases}
\end{equation}
since parabolic equations does not make use of the information before $t_0$, the derivation with respect to $x'$ and $t$ is 0 and $\gamma_1$ is the opposite vector of $\gamma$.

Since $\alpha$ is arbitrary, by varying $\alpha$, we have that $\frac{\partial u}{\partial x_1}<0$ still holds in  
$$\left \{(x,t):x\in \Omega_{\alpha_0}^{+},t_0<t<T \right \}, $$
provided that $\alpha_0$ is small enough. Then we consider $J=u_{x_1}+\varepsilon_1(x_1-\alpha_0)$, where $x\in\Omega_\alpha^+\times(t_0,T)$, $\varepsilon_1=\varepsilon_1(\alpha_0,t_0)>0$ is a constant to be determined later. By direct calculation, we know that 
$$J_t-\Delta J=\frac{\lambda\left[f_{x_1}(1-u)+2f(x)u_{x_1}\right]}{(1-u)^3}\leq 0,$$ 
where $u$ is a solution to (\ref{eq:main equation}). Therefore, $J$ cannot obtain positive maximum value in $\Omega_{\alpha_0}^+\times (t_0,T)$ according to the maximum principle \cite{ref1}.

Since $\frac{\partial u}{\partial x_1}<0$, there is $J=\frac{\partial u}{\partial x_1}$ on $\left \{ x_1=\alpha_0 \right \} $. $\frac{\partial u(x,t_0)}{\partial x_1}$ is also true from $u\ge 0$. If we can prove that $J<0$ on $\Gamma \times (t_0,T)$, then $J<0$ holds on $\Omega_{\alpha_0}^+\times (t_0,T)$, where $\Gamma=\Omega_{\alpha_0}^+\cap \partial \Omega$.

We compare $u$ with the solution $z$ of the heat equation
\begin{equation}
\left\{\begin{matrix}z_t(x,t)-\Delta z(x,t)=0,(x,t)\in\Omega\times(t_0,T),\\ z(x,t)=0, (x,t)\in\partial\Omega\times(t_0,T),\\ z(x,t_0)=0,x\in\Omega.\end{matrix}\right.
\end{equation}

Since $P+\frac{\lambda f(x)}{(1-u)^2}\ge 0$, we have $u-z\ge 0$, then $\frac{\partial u}{\partial \gamma}\leq \frac{\partial z}{\partial \gamma}\leq -C<0$ on $\partial \Omega \times (t_0,T)$, where $C$ is a constant. Therefore, if $x\in \Gamma$, then 
$$J=u_{x_{1}}+\varepsilon_{1}(x_{1}-\alpha_{0})\leq\varepsilon_1(x_1-\alpha_{0})+\frac{\partial u}{\partial\gamma}\frac{\partial\gamma}{\partial x_{1}}\leq\varepsilon_{1}{(x_{1}}-\alpha_{0})-C\frac{\partial \gamma}{\partial x_{1}}<0,$$ 
provided $\varepsilon_{1}$ is small enough, where $\frac{\partial\gamma}{\partial x_1}$ is the derivative of $\gamma$ in $x_1$ direction, and $\frac{\partial u}{\partial x_1}=\nabla u\cdot (1,0,0,...,0)$. 

It can be shown from the maximum principle \cite{ref1} that there exists $\varepsilon_{1}$ small enough such that $J<0$ holds on $\Gamma \times (t_0,T)$, so that $J\leq 0$ holds on $\Omega_\alpha^+\times(t_0,T)$, that is $-u_{x_1}=\left|u_{x_1}\right|\ge\varepsilon_1(x_1-\alpha_0)$. Then when $x'=0$, $\alpha_0\leq x_1<0$, we have 
$$\frac{-u_{x_1}}{(1-u)^2}\ge \frac{\varepsilon_{1}(x_1-\alpha_0)}{(1-u)^2}\ge \varepsilon_{1}(x_1-\alpha_0).$$ 

Integrating with respect to $x_1$, we get that for any $\alpha_0<y_1<0$,
$$-\left(\frac{1}{1-u(y_1,0,t)}-\frac{1}{1-u(\alpha_0,0,t)}\right)\geq\frac{\varepsilon_1}{2}|y_1-\alpha_0|^2,$$ 
that is 
$$\frac1{1-u(y_1,0,t)}\le\frac1{1-u(\alpha_0,0,t)}-\frac{\varepsilon_1}2|y_1-\alpha_0|^2.$$

It follows that 
$$\begin{aligned}
\underset{t\to>T^-}{liminf}\frac{1}{1-u(0,t)}=\underset{t->T^-}{liminf}\underset{y_1-0^-}{\lim}\frac{1}{{1-u(y_1,0,t)}}\\
\leq \underset{t\to T^-}{liminf}\lim\limits_{y_1-0^-}\left[\frac{1}{1-u(\alpha_0,0,t)}-\frac{\varepsilon_1}{2}|y_1-\alpha_0|^2\right].\\
\end{aligned}$$
where $\frac{\varepsilon_1}{2}|y_1-\alpha_0|^2<\infty$ is obvious due to $\left \| u \right \|_{L^\infty (\Omega_T)}<1 $, there is $\left \| \frac{1}{1-u}  \right \|_\infty <\infty $. Therefore, every point in $\{x'=0,\alpha_0\leq x_1<0\}$ is not a quenching point.

Through the above argument, $\alpha_0$ can be independent of the choice of $y_0\in \partial \Omega$. Hence, by varying $y_0\in \partial \Omega$, we conclude that there is an $\Omega$-neighborhood $\Omega'$ of $\partial \Omega$, such that each point $x\in \Omega'$ is not a quenching point. We set $S=\Omega-\Omega'$, then for any $u\in S$, $\frac{1}{1-u}$ blows up. Obviously $S$ is compactly embedded into $\Omega$, then there exists a sequence $\left \{ (x_n,t_{k_n}) \right \} $ of $S$, such that $\lim_{n \to \infty}\frac{1}{1-u(x_n,t_n)}=+\infty $. Then there exists a sub-sequence $\left \{ (x_k,t_{k_j}) \right \} $ such that $\lim_{j \to \infty}\frac{1}{1-u(x_k,t_{k_j})}=+\infty$. Applying the diagonal rule, we can know that there exists $\left \{ (x_k,t_{kk}) \right \} $ such that $\lim_{k \to \infty}\frac{1}{1-u(x_k,t_{kk})}=+\infty$, and $x^*\in S$, where $x^*=\lim_{k \to \infty}x_k$. Therefore, the quenching points lie in a compact subset of $\Omega$ it is clearly a closed set.
\end{proof}

In addition, if $\Omega=B_R(0)$ is a ball of radius $R$ centered at the origin, then according to \cite{ref7} we know that any solution $u(x,t)$ is actually radial symmetric, that is, $u(x,t)=u(r,t)$, $r=|x|\in \left [ 0,R \right ] $. Furthermore, we can prove that the only possible quenching point is exactly the origin.

Our proof requires the following lemma, namely
\begin{lemma}
$u_r<0$ in $\Omega_T\cap \left \{ r>0 \right \} $.
\end{lemma}
\begin{proof}
We set $\Bar{u}=r^{n-1}u_r$, then $\bar{u}_r=(n-1)r^{n-2}u_r+r^{n-1}u_{rr}$, $\bar{u}_t=r^{n-1}\frac{\partial u_t}{\partial r}$, and
$$\Delta u=u_{rr}+\frac{n-1}{r}u_r=\frac{1}{r^{n-1}}\bar{u}_r.$$

Therefore, (\ref{eq:main equation}) becomes $u_t-\Delta u=\frac{\lambda f(x)}{(1-u)^2}+P=u_t-\frac{1}{r^{n-1}}\bar{u}_r$. We differentiate with respect to $r$, and get 
$$\frac{\partial u_t}{\partial r}-\frac{r\frac{\partial\overline u_r}{\partial r}-(n-1)\overline u_r}{r^n}=\frac{2\lambda f(x)u_r}{(1-u)^3}.$$

Then by direct calculation, we have
$$\begin{aligned}
r^{n-1}\frac{\partial u_t}{\partial r}-\frac{\partial \Bar{u}_r}{\partial r}+\frac{(n-1)}{r}\bar{u}_r=\frac{2\lambda f(x)u_{r}}{(1-u)^3}r^{n-1},\\
\bar{u}_{t}-\bar{u}_{rr}+\frac{(n-1)}{r}{\bar{u}_{r}}=\frac{2\lambda f(x)\Bar{u}}{(1-u)^{3}}.
\end{aligned}$$
%$\frac{\partial u_t}{\partial r}-\frac{\partial}{\partial r}\left(\frac{\bar{u}_r}{r^{n-1}}\right)=\frac{2\lambda f(x)u_r}{(1-u)^3}$ 

When $r=|x|=|x_1|$, $u_r=u_{x_1}$, we have that $u_r(r,0)=u_{x_1}(x,0)$ in the positive half axis of $x_1$, $u_r(r,0)=-u_{x_1}(x,0)$ in the negative half axis of $x_1$ by the definition of directional derivative. Since $u(x,t)$ is radially symmetric, there exists $u_{x_1}(x,0)=-u_{x_1}(x,0)$, that is, $u_{x_1}(x,0)=u_r(r,0)=0$. Then we can see that $u_r<0$ holds on $\Omega\cap \left \{ r>0 \right \} $ by the maximum principle \cite{ref1}, since $\Bar{u}=r^{n-1}u_r<0$ on $\partial \Omega \times (0,T)$, and $\Bar{u}_t+L(\Bar{u})\leq 0$.
\end{proof}
\begin{theorem}
Suppose $\Omega=B_R(0)$ and $\frac{\partial f}{\partial n}\leq 0$, if $\lambda>\lambda_P^*$, then the solution only quenches at $r=0$, that is, the origin is the only quenching point.
\end{theorem}
\begin{proof}
We consider $J=\Bar{u}+c(r)F(u)$ as in Theorem 2.3, \cite{ref3}, where $\Bar{u}=r^{n-1}u_r$, $F$ and $c$ are positive functions to be determined, and $F'\ge 0$, $F''\ge 0$. We aim to prove that $J\ge 0$ in $\Omega_T$. Through direct calculation, we have
$$\renewcommand{\arraystretch}{1.5}
\begin{array}{c}
J_t=\bar{u}_t+c(r)F'(u)u_t.\\
J_r=\bar{u}_r+c'(r)F(u)+c(r)F''(u)u_r.\\
J_{rr}=\bar{u}_{r r}+c^{\prime\prime}F+c^\prime F^{\prime}u_{r}+c^\prime F^\prime u_{r}+c(F^{\prime\prime}u_{r}^{2}+F^\prime u_{r}).
\end{array}.$$

So we have
$$\begin{aligned}
J_{t}+\dfrac{n-1}{r}J_{r}-J_{rr}=\bar{u}_{t}+c(r)F'(u)u_{t}+\frac{n-1}{r}(\bar{u}_{r}+c'(r)F(u)+c(r)F'(u)u_{r}\\
-\bar{u}_{rr}+c''F+c'F'u_r+c''F'u_r+c(F''u_r^2+F'u_r)=\frac{2\lambda f(x)}{(1-u)^3}\bar{u}\\
+cF'\left(u_t+\frac{n-1}{r}u_r-u_{rr}\right)+\frac{n-1}{r}c'F-c''F-2c'F'u_r-cF''(u_r)^2.
\end{aligned}$$

Since $\Bar{u}=J-cF$, $\Bar{u}=r^{n-1}u_r$, we know that 
$$\begin{aligned}
J_{t}+\dfrac{n-1}{r}J_{r}-J_{rr}\leq J\left \{ \frac{2\lambda f(x)}{(1-u)^3}+\frac{2(n-1)}{r^n}cF'-\frac{2c'F'}{r^{n-1}}    \right \} \\
-cF\frac{2\lambda f(x)}{(1-u)^3}-\frac{2(n-1)c^2FF'}{r^n}+PcF'+\frac{\lambda f(x)cF'}{(1-u)^2}\\
+\frac{(n-1)c'F}{r}+\frac{2cc'FF'}{r^{n-1}}-c''F.
\end{aligned}$$
Then we have $J_t+\frac{n-1}{r}J_r-J_{rr}=A J+B$, where $A$ is obviously a bounded function for $0<r<R$. 

We choose that $c(r)=\varepsilon r^n$ and $F(u)=\frac{1}{(1-u)^\gamma}$, where $\gamma\ge 0$ is to be determined, then direct calculation yields that
\begin{equation}
\begin{aligned}
B=c(r)\left\{PF'+\dfrac{\lambda f(x)F'}{(1-u)^2}+2\varepsilon FF'-\dfrac{2\lambda Ff(x)}{(1-u)^3}\right\}\quad\\
=c(r)\left\{PF'+2\varepsilon FF'+\dfrac{\lambda f(x)}{(1-u)^3}\left(-2F+F'(1-u)\right)\right\}\quad\\
=c(r)\left\{\dfrac{\gamma-2}{(1-u)^{\gamma+3}}\lambda f(x)+\dfrac{2\gamma\varepsilon}{(1-u)^{2\gamma+1}}+\dfrac{P\gamma}{(1-u)^{\gamma+1}}\right\}.\quad      
\end{aligned}
\end{equation}

If $\gamma\le\min\left\{2-\frac{1}{\lambda c_0},\frac{1}{4\varepsilon},\frac{1}{2P}\right\},\varepsilon\ll 1$, we have
$$\begin{aligned}
\frac{\gamma-2}{(1-u)^{\gamma+3}}{\lambda}f(x)+\frac{2\gamma{\varepsilon}}{(1-u)^{2\gamma+1}}+\frac{P\gamma}{(1-u)^{\gamma+1}}\leq\frac{\gamma-2}{2(1-u)^{\gamma+3}}\lambda f(x)\\
+\frac{2\gamma\varepsilon}{(1-u)^{2\gamma+1}}+\frac{\gamma-2}{{2(1-u)}^{\gamma+3}}{\lambda}f(x)+\frac{P\gamma}{{(1-u)}^{\gamma+1}}\\
\leq \frac{4\gamma\varepsilon-(2-\gamma){\lambda}c_0(1-u)^{\gamma-2}}{2(1-u)^{2\gamma+1}}+\frac{2P\gamma(1-u)^2-(2-\gamma)\lambda c_0}{2(1-u)^{\gamma+3}}\\
\leq 0+0=0,
\end{aligned}$$
then $B\leq 0$.

Since $c(0)=0$, we have $J=\Bar{u}(0,t)=0$, and $J$ cannot obtain positive maximum in $\Omega_T$ or on $\left \{ t=T \right \} $ because $J_t+L(J)=B\leq 0$. Next, we observe by contradiction that if $J_r\leq 0$ on $r=R$, then $J$ cannot obtain positive maximum on $\left \{ r=R \right \}$. 

Then we prove that $J_r\leq 0$ on $\left \{ r=R \right \}$. We have that $J=\bar{u}+c(r)F(u)$ and $J_r=\bar{u}_r+c'(r)F(u)+c(r)F'(u)u_r$, then we have
$$\begin{aligned}
J_r(R,t)=\bar{u}_r(R,t)+c'(R)F\bigl(u(R,t)\bigr)+c(R)F\bigl(\bar{u}(R,t)\bigr)u_r(R,t)\\ 
\leq\bar{u}_{r}(R,t)+c'(R)F\big(u(R,t)\big). 
\end{aligned}$$

Furthermore, $\frac{\lambda f(x)}{(1-u)^2}+P=u_t-\frac{1}{r^{n-1}}\bar{u}_r$, then we can know that if $\varepsilon \ll 1$,
$$\begin{aligned}
J_r(R,t)\leq\bar{u}_r(R,t)+c'(R)F\big(u(R,t)\big)=c'(R)F\Big(u(R,t)\Big)\\
+R^{n-1}\Big(u_t(R,t)-P-\frac{\lambda f(x)}{\big(1-u(R,t)\big)^2}\Big)\\
=R^{n-1}\big(0-P-\lambda f(x) \big)+c'(R)F(0)\\
=R^{n-1}\big(n\varepsilon-P-\lambda f(x)\big)\leq 0.
\end{aligned}$$

Finally, by the maximum principle \cite{ref18}, $\exists t_0\in(0,T),s.t.u_r(r,t_0)<0,0<r\leq R$ and $\exists u_{rr}(0,t_0)<0.$.

Therefore, we have $0\leq r<R$, and provided $\varepsilon\ll 1$, there is $J=\bar{u}+cF$, $J(r,t_0)=r^{n-1}u_r(r,t_0)+\varepsilon r^n F$. And we have $J(r,t_0)<0$ because $0\leq r<R$, $u_r(r,t_0)<0$. By the maximum principle \cite{ref1}, $J\leq 0$ holds for any $t_0 \in (0,T)$ on $B_R\times [t_0,T]$, namely
$$0\ge r^{n-1}u_{r}+cF=r^{n-1}u_r+\varepsilon r^{n}\dfrac{1}{(1-u)^\gamma}=r^{n-1}(u_{r}+\dfrac{\varepsilon r}{(1-u)^{\gamma}})$$
is true for any $0\leq \gamma<1$, so $-u_r(1-u)^\gamma \geq\varepsilon r$ is always true. Integrating with respect to $r$ on $\left [ 0,r \right ] $, we obtain that
\begin{equation}
[1-u(r,t)]^{1+\gamma}-[1-u(0,t)]^{1+\gamma}\ge\dfrac{1}{2}(1+\gamma)\varepsilon r^2.
\end{equation}

It can be seen from Theorem \ref{thm 4.1} that the origin is in the set of quenching points, then $[1-u(r,t)]^{1+\gamma}\ge\frac{1}{2}(1+\gamma)\varepsilon r^2$. If $\forall \  0<r<R, \lim_{t \to T^-}u(r,t)=1  $, then $0\ge \frac{1}{2}(1+\gamma)\varepsilon r^2>0$, which is a contradiction. Therefore, the origin is the only quenching point.
\end{proof}
\section{Quenching behavior.}
\subsection{Upper bound estimate.}
We first obtain a unilateral quenching estimate, namely
\begin{lemma}[Upper bound estimate]\label{lem 5.1}
Let $\Omega \subset R^n$ be a bounded convex domain and $u(x,t)$ be a solution to (\ref{eq:main equation}) which quenches in finite time. Then there exists a bounded positive constant $M>0$ such that $M(T-t)^{\frac{1}{3}}\leq 1-u(x,t)$ holds for all $(x,t)\in \Omega_T$, and when u quenches, $u_t \to +\infty$.
\end{lemma}
\begin{proof}
Since $\Omega$ is a convex bounded domain, we show in Theorem \ref{thm 4.1} that the quenching set of $u$ is a compact subset of $\Omega$. It now suffices to discuss the point $x_0$ lying in the interior domain $\Omega_\eta=\left \{ x\in\Omega:dist(x,\partial \Omega)>\eta \right \} $, for some small $\eta>0$; namely there is no quenching point in $\Omega_\eta^c:=\Omega\setminus \Omega_\eta$.

For any $t_1<T$,  we recall the maximum principle gives $u_t>0$, for all $(x,t)\in \Omega\times (0,t_1)$. Furthermore, the boundary point lemma shows that the exterior normal derivative of $u_t$ on $\partial \Omega$ is negative for $t>0$. This implies that for any small $0<t_0<T$, there exists
a positive constant $C=C(t_0,\eta)$ such that $u_t(x,t_0)\ge C>0$, for all $x\in \Bar{\Omega}_\eta$. For any $0<t_0<t_1<T$, we define a new function,that is
\begin{equation}
J^{\varepsilon}(x,t)=u_t(x,t)-\varepsilon\left(P+\dfrac{\lambda f(x)}{(1-u)^2}\right)=:u_t-\varepsilon\rho(u),
\end{equation}
where $u$ is the solution to (\ref{eq:main equation}). We claim that $J^\varepsilon (x,t)\ge 0$ for all $(x,t)\in \Omega\times (t_0,t_1)$. In fact, it is clear that there exists $C_\eta=C(t_0,t_1,\eta)>0$ such that $u_t\ge C_\eta$ on $\Omega\times (t_0,t_1)$. Furthermore, we can choose $\varepsilon=\varepsilon (t_0,t_1,\eta)>0$ small enough so that $J^\varepsilon (x,t)\ge 0$ on the parabolic boundary of $\Omega_\eta \times (t_0,t_1)$, considering the obvious local boundedness of $P+\frac{\lambda f(x)}{(1-u)^2}$. 

Then by direct calculation, we have that
$$\rho^{\prime}(u)=\frac{\lambda f(x)}{(1-u)^3}, \ \rho^{\prime\prime}(u)=\frac{3\lambda f{(x)}}{(1-u)^4}>0,\  \frac{\partial\rho{(u)}}{\partial t}=\rho^{\prime}{(u)}u_{t}$$
$$\frac{\partial\rho(u)}{\partial x_i}=\rho'(u)u_{x_i},\ \frac{\partial^2\rho(u)}{\partial x_i^2}=\rho''(u)u_{x_i}^2+\rho'(u){u_{x_i^2}}$$
then we have
$$\begin{aligned}
J_{t}^{\varepsilon}(x, t)-\Delta J^{\varepsilon}(x, t)=\frac{\partial}{\partial t}\left(u_{t}-\varepsilon \rho(u)\right)-\Delta\left(u_{t}-\varepsilon \rho(u)\right)\\
=u_{t t}-\varepsilon \rho^{\prime}(u) u_{t}-\Delta u_{t}+\varepsilon \rho^{\prime \prime}(u)|\nabla u|^{2}+\varepsilon \rho^{\prime}(u) \Delta u\\
=\varepsilon \rho^{\prime \prime}(u)|\nabla u|^{2} +\frac{\partial} {\partial t}\left(u_{t}-\Delta u\right)-\varepsilon \rho^{\prime}(u)\left(u_{t}-\Delta u\right)\\
=\varepsilon \rho^{\prime \prime}(u)|\nabla u|^{2}+\rho^{\prime}(u) u_{t}-\varepsilon \rho^{\prime}(u) \rho(u) \\
\geq \rho^{\prime}(u)\left(u_{t}-\varepsilon \rho(u)\right)=\rho^{\prime}(u) J^{\varepsilon}(x,t),
\end{aligned}$$
due to the maximum principle \cite{ref1} and the convexity of $\rho$. Then we have that for any $0<t_0<t_1<T$, there exists $\varepsilon=\varepsilon (t_0,t_1,\eta)>0$ such that $u_t\geq\varepsilon\rho(u)>\varepsilon\frac{\lambda f(x)}{(1-u)^2}$, implying that $u_t \to +\infty$ as $u$ quenches. Furthermore, we integrate it from $t_0$ to $t_1$, and let $t_0\to t^+$, $t_1\to T^-$, then we have 
$$\begin{aligned}
u(x,t_1)\to 1,\ \frac{1}{3}\big(1- u(x,t_0)\big)^3\geq\varepsilon\frac{\lambda f(x)}{(1-u)^2}\geq\varepsilon{\lambda c_0\over(1-u)^2},\\
1-u(x,t)\geq(3\varepsilon\lambda c_0)^{\frac{1}{3}}(T-t)^{\frac{1}{3}}=:M(T-t)^\frac{1}{3}.
\end{aligned}$$

Therefore, there exists $M>0$ such that
\begin{equation}\label{eq Upper bound estimate}
M(T-t)^{\frac{1}{3}}\leq1-u(x,t),
\end{equation}
in $\Omega_\eta\times (0,T)$, due to the arbitrariness of $t_0$ and $t_1$, where $M=M(\lambda,\eta,c_0)$.  Furthermore, one can obtain (\ref{eq Upper bound estimate}) for $\Omega \times (0,T]$,  due to the boundedness of $u$ on $\Omega_\eta^c$.
\end{proof}
\subsection{Gradient estimate.}
Then we study the quenching rate for the higher derivatives of $u$. The idea of the proof is similar to Proposition 1, \cite{ref5} and Lemma 2.6, \cite{ref11}.
\begin{lemma}[Gradient estimate]\label{lem 5.2}
Suppose $\Omega \subset R^n$ is a bounded convex domain and $u(x,t)$ is a quenching solution to (\ref{eq:main equation}) in finite time $T$ for any point $x=a\in \Omega_\eta$, for some small $\eta>0$. Then there exists a positive constant $M'$ such that
\begin{equation}\label{eq Gradient estimate}
|\nabla^mu(x,t)|(T-t)^{-\frac{1}{3}+\frac{m}{2}}\leq M',
\end{equation}
$m=1,2$, holds for $Q_R=B_R\times (T-R^2,T)$, for any $R>0$ such that $a+R\in \Omega_\eta$. 
\end{lemma}
\begin{proof}
It suffices to consider the case $a=0$ by translation. We may focus on some fixed $r$, such that $\frac{1}{2}R^2<r^2<R^2$ and denote $Q_r=B_r\times\left(T\left(1-\left(\frac{r}{R}\right)^2\right),T\right)$.

We first prove that $|\nabla u|$ and $|\nabla^2 u|$ are uniformly bounded on a compact subset of $Q_R$. Indeed, since $P+\frac{\lambda f(x)}{(1-u)^2}$ is bounded on any compact subset $D$ of $Q_R$, standard $L^p$ estimates for heat equations \cite{ref12} give

$$\iint_D\left(|\nabla^2u|^p+|u_t|^p\right)dxdt<c,$$
for $1<p<\infty$ and any cylinder $D$ with $\Bar{D}\subset Q_R$, and $C$ is a generic constant.

We choose $p$ large enough, by the Sobolev embedding theorem \cite{ref1}, we conclude that $u$ is $H\ddot{o} lder$ continuous on $D$, and so is $P+\frac{\lambda f(x)}{(1-u)^2}$. Therefore, Schauder’s estimates for the heat equation \cite{ref12} show that $|\nabla u|$ and $|\nabla^2 u|$ are bounded on any compact subsets of $D$. In particular, there exists $M_1$ such that
$$|\nabla u|+|\nabla^2 u|\leq M_1,$$
for $(x,t)\in B_r\times\left(T\left(1-\left(\frac{r}{R}\right)^2\right),T\left(1-\frac{1}{2}\left(1-\frac{r}{ R}\right)^2\right)\right)$,  where $M_1$ depends on $R$, $n$ and $M$.

We next prove (\ref{eq Gradient estimate}) for $\forall \ (x,t)\in B_r\times\left[T\left(1-\frac{1}{2}\left(1-\frac{r}{R}\right)^2\right),T\right)$. Let us consider
\begin{equation}\label{eq 5.4}
\bar{u}(z,\tau)=1-\mu^{-\frac{2}{3}}\left[1-u\left(x+\mu z,T-\mu^2(T-\tau)\right)\right],
\end{equation}
where $\mu=\left [ 2\left ( 1-\frac{t}{T} \right )  \right ]^\frac{1}{2} $. By direct calculation, we know that 
$$\bar{u}_\tau=\mu^{-\frac{2}{3}}\mu^2u_t=\mu^{\frac{2}{3}+2}u_t,\ \bar{u}_z=\mu^{-\frac{2}{3}}u_x\mu=\mu^{\frac{2}{3}+1}u_x,\ \bar{u}_{zz}=\Delta_z\bar{u}=\mu^{\frac{2}{3}+2}\Delta u.$$
Then we have $\bar{u}_\tau-\Delta_z\bar{u}=\mu^{-\frac{2}{3}+2}(u_t-\Delta u)=\mu^{-\frac{2}{3}+2}(P+\frac{\lambda f(x)}{(1-u)^2})$, and since $1-\Bar{u}=\mu^{-\frac{2}{3}}(1-u)$, we get that
$$\bar u_\tau-\Delta_z\bar u=\mu^{-\frac{2}{3}+2}\left(P+\mu^{-\frac{4}{3}}\frac{\lambda f(x)}{(1-\bar u)^2}\right)=\mu^{-\frac{2}{3}+2}P+\frac{\lambda f(x)}{(1-\bar u)^2}.$$

We also obtain that $\Bar{u}(z,\tau)|_{\partial \Omega} =1-\mu^{-\frac{2}{3}}<0$, and $\Bar{u}(z,0)=\Bar{u}_0(z)=1-\mu^{-\frac{2}{3}}\left[1-u\left(x+\mu z,T(1-\mu^2)\right)\right].$ Therefore, $\Bar{u}$ satisfies
\begin{equation}\label{eq 5.5}
\begin{cases}\bar{u}_\tau-\triangle_z\bar{u}=P\mu^{-\frac{2}{3}+2}+\frac{\lambda f(x)}{(1-\bar{u} )^2},\quad(z,\tau)\in\mathcal{O}_T,\\ \bar{u}(z,\tau)=1-\mu^{-\frac{-2}{3}}<0,\quad(z,\tau)\in\partial\mathcal{O}_{T},\\\bar{u}(z,0)=\bar{u}_0(z),\quad z\in\mathcal{O}.\end{cases}
\end{equation}
where $\bar{u}_0(z)=1-\mu^{-\frac{2}{3}}[1-u(x+\mu z,T(1-\mu^2))]<0$, and since $u_t(x,T)$ is non-negative, we have that
$$\Bar{u}_\tau (z,0)=\mu^{-\frac{2}{3}+2}u_t(x,T)=\Delta_z\bar u(z,0)+P\mu^{-\frac23+2}+\frac{\lambda f(x)}{(1-\bar u)^2}\ge 0.$$

For the fixed point $(x,t)$, we define $\mathcal{O}:=\{z:x+\mu z\in\Omega\}$. It is implied by (\ref{eq 5.4}) that $T$ is also the finite quenching time of $\Bar{u}$, and when $t\to T$, $u\to 1$ and $\Bar{u}\to 1$. Furthermore, we claim that the domain of $\Bar{u}$ includes $Q_{r_0}$ for some $r_0=r_0(R)>0$. The claim is equivalent to prove that there exists $r_0=r_0(R)$, such that $\tau$ is contained in $\left(T-T\left(\frac{r_0}{R}\right)^2,T\right)$. Because $t=T-\mu^2(T-\tau)<T$, $\tau<T$ is obvious. To prove $\tau \ge T-T(\frac{r_0}{R})^2$, we just need to prove $T-\tau\leq T(\frac{r_0}{R})^2$, and also because $T-\mu^2(T-\tau)\geq T-\frac{T}{2}\left(1-\frac{r}{R}\right)^2$, we have
$$\begin{aligned}
\mu^2(T-\tau)\le\dfrac{T}{2}\left(1-\dfrac{r}{R}\right)^2\Rightarrow T-\tau\le\dfrac{T}{2\mu^2}\left(1-\dfrac{r}{R}\right)^2\\
\Rightarrow T-\tau\le\dfrac{T}{4}\left(1-\dfrac{r}{R}\right)^2\le\dfrac{T}{2\mu^2}\left(1-\dfrac{r}{R}\right)^2,
\end{aligned}$$
since $\mu\leq \sqrt{2}$. Then we just need to find some $r_0$ satisfying $r_0\ge \frac{R}{2} \frac{R-r}{R}=\frac{R-r}{2}$, which is obviously true.

Since the quenching set of $u$ is a compact subset of $\Omega$, so is that of $\Bar{u}$ just by translation. Therefore, the argument of Lemma \ref{lem 5.1} can be applied to (\ref{eq 5.5}), yielding that there exists a constant $M_2>0$ such that
$$1-\bar{u}(z,\tau)\geq M_2(T-\tau)^{\frac{1}{3}},$$
where $M_2$ depends on $R, \lambda, P \ and\  \Omega$.

Applying the interior $L^p$ estimates and Schauder’s estimates \cite{ref12} to $\Bar{u}$ as before, then there exists $M_1'=M_1'(R,\lambda,P,n,M_2)>0$ such that
\begin{equation}\label{eq 5.6}
|\nabla_z\bar{u}|+|\nabla_z^2\bar{u}|\leq M_1',
\end{equation}
for $(z,\tau)\in B_r\times\left(T\left(1-\left(\frac{r}{r_0}\right)^2\right),T\left(1-\frac{1}{2}\left(1-\frac{r}{r_{0}}\right)^2\right)\right)$, where we assume that $\frac{1}{2}r_0^2<r^2<r_0^2$. We take $(z,\tau)=(0,\frac{T}{2})$ in (\ref{eq 5.4}), then we have $\mu^{-\frac{2}{3}+1}|\nabla_x u|+\mu^{-\frac23+2}|\nabla^2_x u|\leq M_1'$ since $\nabla_{z} \bar{u}=\mu^{-\frac{2}{3}+1} \nabla_{x} u$ and $\nabla_{z}^{2} \bar{u}=\mu^{-\frac{2}{3}+2} \nabla_{x}^{2} u$. Thus, from $\mu=\left [ 2\left ( 1-\frac{t}{T} \right )  \right ]^\frac{1}{2}$ we have
$$\begin{aligned}
\left[2\left(1-\frac{t}{T}\right)\right]^{-\frac{1}{3}+\frac{1}{2}}|\nabla_x u|+\left[2\left(1-\frac{t}{T}\right)\right]^{-\frac{1}{3}+1}|\nabla_x^2u|\\
= \mu^{-\frac{2}{3}+1}|\nabla_x u|+\mu^{-\frac{2}{3}+2}|\nabla_x^2u|\leq M_1'.
\end{aligned}$$
Since $\frac{2}{T}$ is a constant, we have Lemma \ref{lem 5.2} immediately.
\end{proof}
\subsection{Lower bound estimate.}
\begin{proposition}[Lower bound estimate]\label{prop 5.3}
Suppose $\Omega \subset R^n$ is a bounded convex domain and $u(x,t)$ is a quenching solution to (\ref{eq:main equation}) in finite time $T$, then there exists a bounded constant $C=C(\lambda,\Omega)>0$ such that 
\begin{equation}\label{eq Lower bound estimate}
\max\limits_{x\in\Omega}u(x,t)\ge1-C(T-t)^{\frac{1}{3}},t\to T^-,
\end{equation}
for $t\to T^-$.
\end{proposition}
\begin{proof}
Let $U(t) = \max\limits_{x \in \Omega} u(x,t), 0 < t < T$, and let $U(t_i)=u(x_i,t_i),i=1,2$, with $h=t_2-t_1>0$. Then
$$U(t_2)-U(t_1)=u(x_2,t_2)-u(x_1,t_1)\geq u(x_1,t_2)-u(x_1,t_1)=hu_t(x_1,t_2)+o(1),$$
 and
 $$U(t_2)-U(t_1)=u(x_2,t_{2})-u(x_{1},t_{1})\leq u(x_{2},t_{2})-u(x_{2},t_{1})=hu_1(x_{2},t_2)+o(1).$$
It follows that $U(t)$ is Lipschitz continuous. Hence, for $t_2>t_1$, we have 

$$\frac{U(t_2)-U(t_1)}{t_2-t_1}\leq u_t(x_2,t_2).$$

On the other hand, since $\nabla u(x_2,t_2)=0$ and $\bigtriangleup u(x_2, t_2)\le0$, and near the quenching point, $u(x,t)\to 1$, then $\forall \varepsilon>0$, $P\leq \frac{\varepsilon \lambda f(x)}{(1-u)^2}$, then we have 
$$\begin{aligned}
u_{t}\left(x_{2}, t_{2}\right)=P+\frac{\lambda f\left(x_{2}\right)}{\left(1-u\left(x_{2}, t_{2}\right)\right)^{2}}+\Delta u\left(x_{2}, t_{2}\right) \leq P+\frac{\lambda f\left(x_{2}\right)}{\left(1-u\left(x_{2}, t_{2}\right)\right)^{2}} \\
\leq(1+\varepsilon)\frac{\lambda f(x_2)}{(1-u(x_2,t_2))^2}
\end{aligned}$$
for all $0<t_2<T$.

Because $f(x)$ is bounded on $\Bar{\Omega}$, $|f(x)|\leq M$, $(1-U)^2U_t\leq (1+\varepsilon)\lambda f(x)\leq (1+\varepsilon)\lambda M$ holds for a.e. $0<t<T$. Integrate from $t$ to $T$, we obtain 
$${\frac{1}{3}}(1-U(t))^3-\frac{1}{3}(1-U(T))^3\le\lambda M(1+\varepsilon)(T-t),$$ 
namely 
$$(1-U(t))^3\leq3\lambda M(1+\varepsilon)(T-t).$$

Let $\varepsilon\to 0^+$, we have $(1-U(t))^3\leq3\lambda M(T-t)=:C^3(T-t)$, that is $U(t)\ge 1-C(T-t)^\frac{1}{3}$, then $\max\limits_{x\in\Omega}u(x,t)\ge1-C(T-t)^{\frac{1}{3}}$ as $t\to T^-$.
\end{proof}
\subsection{Nondegeneracy of quenching solution.}
For the quenching solution $u(x,t)$ of (\ref{eq:main equation}) in finite time $T$, we now introduce the associated similarity variables
\begin{equation}\label{eq parameter transformation}
y=\dfrac{x-a}{\sqrt{T-t}},s=-\log(T-t),u(x,t)=1-(T-t)^{\frac{1}{3}}w_a(y,s), 
\end{equation}
where $a$ is any point in $\Omega_\eta$, for some small $\eta>0$. The form of $w_a$ defined in (\ref{eq parameter transformation}) is motivated by (\ref{eq Upper bound estimate}) and (\ref{eq Lower bound estimate}). Then $w_a(y,s)$ is defined in
$$W_a:=\{(y,s):a+ye^{-\frac{s}{2}}\in\Omega,s>s'=-\log T\}.$$

By differentiating $y$ and $s$ respectively, we can easily obtain $w_a(y,s)$ solves
\begin{equation}\label{eq non divergence form}
\dfrac{\partial w_a}{\partial s}=\dfrac{1}{3}w_a-\dfrac{1}{2}y\nabla w_a+\Delta w_a-\dfrac{\lambda f\left(a+ye^{-\frac{s}{2}}\right)}{w_a^2}-Pe^{-\frac{2s}{3}}.  
\end{equation}

Here $w_a$ is always strictly positive in $W_a$. The slice of $W_a$ at a given time $s=s_0$ will be denoted as $\Omega_a(s_0)$:
$$\Omega_a(s_0):=W_a\cap\{s=s_0\}=e^{\frac{s_0}{2}}(\Omega-a).$$

For any $a\in\Omega_\eta$, there exists $s_0=s_0(\eta,a)>0$ such that
\begin{equation}
B_s:=\left \{ y:|y|<s  \right \}\subset \Omega_a(s)  
\end{equation}
for $s\ge s_0$.

Equation (\ref{eq non divergence form}) could also be written in divergence form:
\begin{equation}\label{eq divergence form}
\rho w_s=\nabla(\rho\cdot\nabla w)+\dfrac{1}{3}\rho w-\frac{\lambda \rho f(a)}{w^2}-\rho Pe^{-\frac{2}{3}s},  
\end{equation}
with $\rho (y)=e^{-\frac{|y|^2}{4}}$.

We shall research the non-degeneracy of the quenching behavior. The conclusion is obtained by the comparison principle \cite{ref4} and results in \cite{ref11}.

\begin{theorem}[Nondegeneracy of quenching solution.]\label{thm 5.4}
Suppose $\Omega \subset R^n$ is a bounded convex domain and $u(x,t)$ is a quenching solution to (\ref{eq:main equation}) in finite time $T$ for any point $x=a\in \Omega_\eta$, for some $\eta>0$. If $w_a(y,s)\to \infty$ as $s\to \infty$ uniformly for $|y|\leq C$, where C is any positive constant, then $a$ is not a quenching point of $u$.
\end{theorem}
\begin{proof}
It is easy to see that $w_a$ in (\ref{eq non divergence form}) is a sub-solution to
\begin{equation}\label{eq Nondivergence equation in limit sense}
\frac{\partial\widetilde w}{\partial s}=\frac{1}{3}\widetilde w-\frac{1}{2}y\nabla\widetilde w+\Delta\widetilde w-\frac{\lambda f(a)}{\widetilde w^2}
\end{equation}
in $B_{s_0}\times (s_0,\infty)$. From the comparison principle \cite{ref4}, we get $w_a\leq \widetilde w$ in $B_{s_0}\times (s_0,\infty)$. If $w_a(y,s)\to \infty$ as $s\to \infty$ uniformly for $|y|\leq C$, so does $\widetilde w$. If $f(a)\ge 1$, we let $|f(a)|\leq 1$ through a scaling transformation $\lambda\to \frac{\lambda}{M}$, where $M$ is the upper bound of $f$ on $\Bar{\Omega}$. Then our conclusion follows immediately from Theorem 2.12, \cite{ref11}, where $\widetilde w$ is the $w_a$ in \cite{ref11}.
\end{proof}
\subsection{Asymptotics of quenching solution.}
In this subsection, we shall omit all the subscription $a$ of $w_a$, $W_a$ and $\Omega_a$ if no confusion will arise. In view of (\ref{eq parameter transformation}), one combines Lemma \ref{lem 5.1} and Lemma \ref{lem 5.2} to reach the following estimates on $w$, $\nabla w$ and $\Delta w$:
\begin{proposition}\label{prop 5.5} 
Suppose $\Omega \subset R^n$ is a bounded convex domain and $u(x,t)$ is a quenching solution to (\ref{eq:main equation}) in finite time $T$. Then the rescaled solution $w$ satisfies
\begin{equation}
M\le w\le e^{\frac{s}{3}},\quad|\nabla w|+|\triangle w|\le M',\ \text{in}\ W,
\end{equation}
where $M$ and $M'$ are constants in Lemma \ref{lem 5.1} and Lemma \ref{lem 5.2}, respectively. Moreover, it satisfies
\begin{equation}
M\leq w(y_1,s)\leq w(y_2,s)+M'|y_1-y_2|,
\end{equation}\label{eq mean value theorem for w}
for any $(y_i,s)\in W,i=1,2$.
\end{proposition}
\begin{proof}
On the one hand, $M(T-t)^\frac{1}{3}\leq (T-t)^\frac{1}{3}w_a$, and $T-t=e^{-s}$, $u(x,t)=1-(T-t)^\frac{1}{3}w_a(y,s)\ge 0$, then $w_a\leq (T-t)^\frac{1}{3}=e^\frac{s}{3}$. On the other hand, $w(y_1)-w(y_2)\leq |y_1-y_2|w'(\xi)\leq M'|y_1-y_2|$.
\end{proof}
\begin{lemma}\label{lem 5.6}
Let $s_j$ be an increasing sequence such that $s_j\to +\infty$, and $w(y,s+s_j)$ is uniformly convergent to a limit $w_\infty (y,s)$ in compact sets. Then either $w_\infty (y,s)\equiv \infty$ or $w_\infty (y,s)<\infty$ in $R^n$.
\end{lemma}
\begin{proof}
(\ref{eq mean value theorem for w}) implies that $$w_\infty (y_1,s)\leq w_\infty (y_2,s)+M'|y_1-y_2|,$$
and if $w_\infty (y,s)<\infty$, then the conclusion follows; if there exists $y_0$ such that $w_\infty (y_0,s)=\infty$, then $w_\infty (y_0,s)\leq w_\infty (y_1,s)+M'|y_1-y_0|$, the left side is $\infty$ but the right side not, which is a contradiction.
\end{proof}
\begin{theorem}\label{thm 5.7}
Let $\Omega \subset R^n$ be a bounded convex domain and $u(x,t)$ be a solution to (\ref{eq:main equation}) which quenches in finite time, then $w_a(y,s)\to w_\infty (y)$, as $s\to \infty$ uniformly on $|y|\leq C$, where $C>0$ is any bounded constant, and $w_\infty (y)$ is bounded positive solution to
\begin{equation}\label{eq 5.15}
    \triangle w-\dfrac12y\cdot\nabla w+\dfrac13w-\frac{\lambda f(a+ye^{-\frac{s}{2}})}{w^2}=0
\end{equation}
in $R^n$. Moreover, if $\Omega\in R$ or $\Omega\in R^n$, $n\ge 2$, is a convex bounded domain, then we have
$$\lim\limits_{t\to T^-}(1-u(x,t))(T-t)^{-\frac{1}{3}}=(3\lambda f(a))^{\frac{1}{3}}$$
uniformly on $|x-a|\leq C\sqrt{T-t}$ for any bounded constant C.
\end{theorem}
\begin{proof}
Let us adapt the arguments in the proofs of Propositions 6 and 7, \cite{ref5} or Lemma 3.1, \cite{ref11}. 

Let $\left \{ s_j \right \} $ be an increasing sequence tending to $\infty$ and $s_{j+1}-s_j\to \infty$. We denote $w_j(y,s)=w(y,s+s_j)$ and $z_j(y,s)=\frac{1}{w_j(y,s)}$. Since $M\leq w\leq e^{\frac{s}{3}}$, $z_j\leq \frac{1}{M}$ is uniformly bounded, and
$$w_j(y_1,s)-w_j(y_2,s)\leq M'|y_1-y_2|,\forall j,\forall y_1,y_2,$$
then $\left \{ z_j \right \} $ is also equi-continuous. Applying the Arzela-Ascoli theorem on $z_j(y,s)=\frac{1}{w_j(y,s)}$ with 
Proposition \ref{prop 5.5}, there is a subsequence of ${z_j}$, still denoted as ${z_j}$, such that
$$z_j(y,s)\to z_\infty (y,s)$$
uniformly on compact sets of $W$. Furthermore, we have $\nabla z_j(y,m)$ is continuous and uniformly converges since $\left|\nabla z_j(y,m)\right|=\left|-\frac{\nabla w_j}{\left|w_j\right|^2}\right|\leq\frac{M'}{M^2}$ and Weierstrass Discriminance, then $\nabla z_j(y,m)\to\nabla z_\infty(y,m)$ for almost all $y$ and for each integer $m$.

Therefore, $w_j(y,s)\to w_\infty (y,s)$ uniformly on the compact sets of $W$ and $\nabla w_j(y,m)\to\nabla w_\infty(y,m)$ for almost all $y$ and for each integer $m$. From Lemma \ref{lem 5.6}, we get that either $w_\infty (y,s)\equiv \infty$ or $w_\infty (y,s)< \infty$ in $R^n$. Since $a$ is the quenching point, the case of $w_\infty (y,s)\equiv \infty$ is not true by 
Theorem \ref{thm 5.4}, then $w_\infty (y,s)< \infty$.

We define the associate energy of $w$ at time $s$:
\begin{equation}\label{eq associate energy}
E[w](s)=\dfrac12\int_{B_s}\rho|\nabla w|^2dy-\dfrac16\int_{B_s}\rho w^2dy-\lambda\int_{B_s}\dfrac{\rho f(a)}{w}dy.
\end{equation}

Direct calculation yields that
\begin{equation}\label{eq 5.16}
 \begin{aligned}
\frac{d}{d s} E[w](s)= & \int_{B_{s}} \rho \nabla w \cdot \nabla w_{s} d y-\frac{1}{3} \int_{B_{s}} \rho w w_{s} d y+\lambda \int_{B_{s}} \frac{\rho f(a)}{w^{2}} w_{s} d y \\
& +\frac{1}{2} \int_{\partial B_{s}} \rho|\nabla w|^{2}(y \cdot \nu) d S-\frac{1}{6} \int_{\partial B_{s}} \rho w^{2}(y \cdot \nu) d S \\
& -\lambda \int_{\partial B_{s}} \frac{\rho f(a)}{w}(y \cdot \nu) d S \\
& =-\int_{B_{s}} \nabla(\rho \cdot \nabla w) w_{s} d y-\frac{1}{3} \int_{B_{S}} \rho w w_{s} d y \\
& +\lambda \int_{B_{s}} \frac{\rho f(a)}{w^{2}} w_{s} d y+\int_{\partial B_{s}} \rho(\nu \cdot \nabla w) w_{s} d S \\
& +\frac{1}{2} \int_{\partial B_{s}} \rho|\nabla w|^{2}(y \cdot \nu) d S-\frac{1}{6} \int_{\partial B_{s}} \rho w^{2}(y \cdot \nu) d S \\
& -\lambda \int_{\partial B_{s}} \frac{\rho f(a)}{w}(y \cdot \nu) d S \\
& =G(s)-\int_{B_{s}} \rho\left|w_{s}\right|^{2} d y-\int_{B_{s}} \rho w_{s} P e^{-\frac{2}{3} s} dy,
\end{aligned}
\end{equation}
where
\begin{equation}\label{eq G(s)}
\begin{aligned}
G(s)=\int_{\partial B_{s}} \rho(\nu \cdot \nabla w) w_{s} d S+\frac{1}{2} \int_{\partial B_{s}} \rho|\nabla w|^{2}(y \cdot \nu) d S \\
-\frac{1}{6} \int_{\partial B_{s}} \rho w^{2}(y \cdot \nu) d S-\lambda \int_{\partial B_{s}} \frac{\rho f(a)}{w}(y \cdot \nu) d S,
\end{aligned}
\end{equation}
where  $\nu$ is the exterior unit normal vector to $\partial \Omega$ and $dS$ is the surface area element. The first equality in (\ref{eq 5.16}) is followed by Lemma 2.3, \cite{ref15}. Then we estimate $G(s)$ as in Lemma 2.10, \cite{ref11}:

Considering Lemma \ref{lem 5.6} and the fact that $a$ is the quenching point, we have 
\begin{equation}\label{eq 5.18}
\begin{aligned}
|w_s|=\left|\frac13w-\frac12y\nabla w+\Delta w-\frac{\lambda f(a)}{w^2}-Pe^{-\frac{2}{3}s}\right|\leq \\
\left|\frac13w\right|+\left|\frac{1}{2}y\nabla w\right|+|\Delta w|+\left|\frac{\lambda f(a)}{w^2}\right|+\left|Pe^{-\frac{2}{3}s}\right|\leq\\
\frac{w}{3}+\frac{M'}{2}|y|+M'+\frac{\lambda_p^2\max f(a)}{M^2}+P^*\\
\leq C(1+|y|)+\frac{w}{3}\leq \tilde{C}(s+1),
\end{aligned}
\end{equation}
the last inequality holds since $|y|$ and $w$ are bounded. 

Then $G(s)$ can be estimated as follow:
\begin{equation}\label{eq 5.19}
\begin{aligned}
G(s)\leq\int_{\partial B_s}\rho(& (\nu\cdot\nabla w)w_sdS+\frac{1}{2}\int_{\partial B_s}\rho|\nabla w|^2(y\cdot\nu)dS  \\
&\leq C_1s^n e^{-\frac{s^2}{4}}+C_2s^{n-1}e^{-\frac{s^{2}}{4}}\lesssim s^{n}e^{-\frac{S^2}{4}},
\end{aligned}
\end{equation}
since (\ref{eq 5.18}) and
$$\begin{aligned}
\int_{\partial B_s}\rho(\nu\cdot\nabla w)w_sdS\leq\tilde{C}\int_{\partial B_s}(1+s)e^{-\frac{|y|^2}{4}}(\nu\cdot{\nabla w})dS\\
=\tilde{C}e^{-\frac{s^2}{4}}\int_{\partial B_S}(1+s)(\nu\cdot\nabla w)dS\\
\leq C'e^{-\frac{s^2}{4}}\int_{\partial B_s}(1+s)dS\leq C_1s^n e^{-\frac{s^3}{4}},
\end{aligned}$$
considering $|\nabla w|$ and $(\nu \cdot \nabla w)$ are bounded. Hence, by integrating (\ref{eq 5.16}) in time from $a$ to $b$, we have that
\begin{equation}\label{eq 5.20}
\begin{aligned}
\int_a^b\int_{B_s}\rho\left | w_s \right |^2 dyds=E[w](a)-E[w](b)+\tilde{C}\int_a^b G(s)ds\\
-\int_a^b\int_{B_s}\rho w_sPe^{-\frac{2}{3}s}dyds\leq E[w](a)-E[w](b)+\tilde{C}\int_a^b G(s)ds \\
+C\int_{a}^{b}se^{-\frac{2s}{3}}ds\int_{B_s}\rho dy,
\end{aligned}
\end{equation}
for any $a<b$. Now we shall show that $w_\infty$ is independent of $s$. We set $a=m+s_j$, $b=m+s_{j+1}$, and $w=w_j$ in (\ref{eq 5.20}), then we have
\begin{equation}\label{eq 5.21}
\begin{aligned}
&\int_m^{m+s_{j+1}-s_j}\int_{B_{s+s_j}}\rho\left|\frac{\partial w_j}{\partial s}\right|^2dyds \\
&\le E\bigl[w_j\bigr](m)-E\bigl[w_{j}\bigr](m)+\tilde{C}\int_{m+s_j}^{m+s_{j+1}}G(s)ds \\
&+C\int_{m+s_j}^{m+s_{j+1}}se^{-\frac{2s}{3}}ds\int_{B_s}\rho dy,
\end{aligned}
\end{equation}
for any integer $m$. Since $s_j+m\to \infty$ as $j\to \infty$, considering (\ref{eq 5.19}), the third term on the right-hand side of (\ref{eq 5.21}) tends to zero. Furthermore, by direct computation, we have $\int_a^b se^{-\frac{2s}{3}}ds=\left(\frac{2}{3}a+\frac{9}{4}\right)e^{-\frac{2}{3}a}-\left(\frac{2}{3}b+\frac{9}{4}\right)e^{-\frac{2}{3}b}\to 0$ as $j\to \infty$.

Since $\nabla w_j(y,m)$ is bounded and independent of $j$, and $\nabla w_j(y,m)\to \nabla w_\infty (y,m)$ a.e. as $j\to \infty$, we have
\begin{equation}\label{eq 5.22}
\operatorname*{lim}_{j\rightarrow\infty}E[w_{j}](m)=\operatorname*{lim}_{j\rightarrow\infty}E[w_{j+1}](m):=E[w_{\infty}],
\end{equation}
by the dominated convergence theorem. Thus, the right-hand side of (\ref{eq 5.21}) tends to zero as $j\to \infty$. Therefore
\begin{equation}\label{eq 5.23}
\operatorname*{lim}_{j\to\infty}\int_{m}^{M}\int_{B_{s+s_{j}}}\rho\left|\frac{\partial w_{j}}{\partial s}\right|^{2}d y d s=0,
\end{equation}
for each pair of $m$ and $M$. Therefore, we can get $\frac{\partial w_j}{\partial s}$ converges weakly to $\frac{\partial w_\infty}{\partial s}$. And by definition, we can prove that the integral in (\ref{eq 5.23}) is lower semi-continuous. Thus we conclude that 
$$\int_m^M\int_{R^n}\left|\frac{\partial\text{w}_\infty}{\partial s}\right|^2dyds=0,$$ 
since 
$$\begin{aligned}
\lim\limits_{j\to\infty}\int_m^M\int_{B_{s+s_j}}\rho\left|\frac{\partial w_j}{\partial s}\right|^2dyds=\liminf\int_m^{M}\int_{B_{s+s_j}}\rho\Big|\frac{\partial w_j}{\partial s}\Big|^2dyds=0\\
\geq\int_m^M\int_{B_{s+s_j}}\Big|\frac{\partial w_\infty}{\partial s}\Big|^{2}dyds\geq 0,
\end{aligned}$$
according to the Fatou lemma. Then $w_\infty$ is independent of $s$ because $m$ and $M$ are arbitrary.

Since $|\frac{\partial w_j}{\partial s}|$ and $\nabla w_j$ are locally bounded in $R^n\times (s_0,\infty)$ for some $s_0\gg 1$, by Proposition \ref{prop 5.5}, $w_\infty$ is locally Lipschitzian. Each $w_j$ solves (\ref{eq non divergence form}), and $\lim _{s \rightarrow+\infty} a+y e^{-\frac{s}{2}}=a$,  $ \lim_{s \to \infty}  \rho P e^{-\frac{2}{3} s}=0$, so $w_\infty$ is a stationary weak solution to (\ref{eq 5.15}). Then 
$$\widetilde{w}^2\frac{\partial\tilde{w}}{\partial s}=\frac{1}{3}\widetilde{w}^3-\frac{1}{2}\widetilde{w}^{2}y\nabla\widetilde{w}+\widetilde{w}^2 \Delta \widetilde{w} - \lambda f(a),$$
and $\frac{1}{2}\widetilde{w}^2 y \nabla \widetilde{w}$ satisfy the uniform elliptic condition, then there exists a unique classical solution \cite{ref4}, that is, $w_\infty$ is actually a classical solution.

The solution to (\ref{eq Nondivergence equation in limit sense}) in one dimension has been investigated in \cite{ref2}. And \cite{ref10} studied the solution to this equation of high dimension. Combining their results, we get the following:

It is shown in Theorem 2.1, \cite{ref2} and Theorem 1.6, \cite{ref10} that every non-constant solution $w(y)$ to (\ref{eq Nondivergence equation in limit sense}) in $R^n$ must be strictly increasing for sufficiently large $|y|$, and $w(y)\to \infty$, as $|y|\to \infty$. If $f(a)\ge 1$, we let $|f(a)|\leq 1$ through a scaling transformation $\lambda\to \frac{\lambda}{M}$, where $M$ is the upper bound of $f$ on $\Bar{\Omega}$. Therefore, $w_\infty$ has to be a constant solution, namely $w_\infty\equiv (3\lambda f(a))^\frac{1}{3}$.
\end{proof}

\section{Conclusion.}
In this paper, we study the (\ref{eq:main equation}) modelling the MEMS device with the pressure term $P>0$ in high dimensions. We first show that if $0\leq P\leq P^*, 0<\lambda\leq \lambda_P^*$, there exists a unique global solution $u(x,t)$ to (\ref{eq:main equation}); otherwise, the solution will quench in finite time $T<\infty$. 

About the quenching time $T$, for $\lambda>\lambda_P^*$, we show that it satisfies $T \lesssim \frac{1}{\lambda}$. According to the comparison principle, an upper bound for the pull-in voltage $\lambda_P^*$ and the pressure term $P^*$ is proved that $\lambda_P^*\leq \lambda_1:=\frac{\mu_0-P}{c_0}$ and $P^*\leq \mu_0$.

By adapting the moving-plane argument as in \cite{ref7}, we show that the quenching set of (\ref{eq:main equation}) is a compact set in $\Omega$ provided that $\frac{\partial f}{\partial n}\leq 0$, where $\Omega \subset R^n$ is a bounded convex set. Furthermore, if $\Omega=B_R(0)$, the ball centered at the origin with the radius $R$, then the origin is the only quenching point.

Finally, we study the quenching behavior of the solution to (\ref{eq:main equation}). In this paper that, under certain conditions, we derive the upper bound estimate, the gradient estimate and the lower bound estimate of the quenching rate, and obtain the asymptotic behavior of the quenching solution in finite time, that is, 
$$\lim\limits_{t\to T^-}(1-u(x,t))(T-t)^{-\frac{1}{3}}=(3\lambda f(a))^{\frac{1}{3}}.$$

\section*{Acknowledgement.}This work is partially supported by the NSF of China 11871148, 11671079, 11101078, and the NSF of Jiangsu Province BK20161412.
%% The Appendices part is started with the command \appendix;
%% appendix sections are then done as normal sections
%% \appendix

%% \section{}
%% \label{}

%% If you have bibdatabase file and want bibtex to generate the
%% bibitems, please use
%%
%%  \bibliographystyle{elsarticle-harv} 
%%  \bibliography{<your bibdatabase>}

%% else use the following coding to input the bibitems directly in the
%% TeX file.

\bibliographystyle{elsarticle-harv} 
\bibliography{ref}
\end{document}